\documentclass[10pt,reqno]{amsart}

\usepackage{amssymb}
\usepackage{graphicx,color}
\usepackage{hyperref}
\usepackage{esint}
\usepackage{tkz-berge}
\usepackage{subfig}
\usepackage{cancel}

     \makeatletter
     \def\section{\@startsection{section}{1}%
     \z@{.7\linespacing\@plus\linespacing}{.5\linespacing}%
     {\bfseries
     \centering
     }}
     \def\@secnumfont{\bfseries}
     \makeatother
\newtheorem{theorem}{Theorem}[section]
\newtheorem{lemma}[theorem]{Lemma}
\newtheorem{proposition}[theorem]{Proposition}
\newtheorem{corollary}[theorem]{Corollary}
\theoremstyle{definition}
\newtheorem{definition}[theorem]{Definition}

\theoremstyle{remark}
\newtheorem{remark}[theorem]{Remark}
\numberwithin{equation}{section}
\begin{document}

\title{Totalitarian Random Tug-of-War games in graphs}

\author[Marcos Ant\'on]{Marcos Ant\'on}
\address{Universitat Polit\`ecnica de Catalunya, Departament
de Matem\`{a}tiques, Diagonal 647, 08028 Barcelona, Spain}
\email{m.anton.a@outlook.com}

\author[Fernando Charro]{Fernando Charro}
\address{Fernando Charro: 
Department of Mathematics, Wayne State University, Detroit, MI 48202, USA}
\email{fcharro@math.wayne.edu}

\author[Peiyong Wang]{Peiyong Wang*}
\thanks{* Peiyong Wang is partially supported by a Simon's Collaboration Grant.}
\address{Peiyong Wang: 
Department of Mathematics, Wayne State University, Detroit, MI 48202, USA}
\email{pywang@math.wayne.edu}




\begin{abstract}
In this work we discuss a random Tug-of-War game in graphs where one of the players has the power to decide at each turn whether to play a round of classical random Tug-of-War, or let the other player choose the new game position in exchange of a fixed payoff. We prove that this game has a value using a discrete comparison principle and viscosity tools, as well as probabilistic arguments. This game is related to Jensen's extremal equations, which have a key role in Jensen's celebrated proof of uniqueness of infinity harmonic functions. 
\end{abstract}

\maketitle



\section{Introduction}\label{section.intro}

Random Tug-of-War games were introduced in \cite{Peres2009} in connection with partial differential equations (see also the survey \cite{Rossi2010}). Informally, random Tug-of-War games play for the normalized infinity Laplacian
\begin{equation}\label{infinityLaplacian_normalized}
\Delta_\infty^N u(x):=
\begin{cases}
\left<D^2u(x)\, \frac{\nabla u(x)}{|\nabla u(x)|},\frac{\nabla u(x)}{|\nabla u(x)|}\right>, & \textnormal{if}\ \nabla u(x)\neq0; \\[0.5em]
\lim_{y\to x}\frac{2(u(y)-u(x))}{|y-x|^2}, & \textnormal{otherwise}
\end{cases}
\end{equation}the role that the Brownian motion plays for the Laplacian. Observe that \eqref{infinityLaplacian_normalized} is the pure second derivative of $u$ in the direction of the gradient whenever $\nabla u(x)\neq0$.
On the other hand, at points where $\nabla u(x)=0$ no direction is preferred and it is only required that the limit exists.

\subsection{Classical random Tug-of-War games}\label{introd.classical.ToW}
The classical random Tug-of-War game (see \cite{Peres2009}) is a two-person, zero-sum game, that is, two players are in contest and the total earnings of one player are the losses of the other. 

Following \cite{Peres2009}, random Tug-of-War games can be described in a very general way in terms of a set $X$ of states of the game, a non-empty set $Y$ of terminal states, and an undirected graph $E$ with vertex set $X\cup Y$ that describes the possible move options for both players at any game state.

The game starts with a token placed at $x^0\in X\backslash Y$ and is played by turns. At each turn a fair coin is tossed and the winner of the toss is allowed to decide the next game position among all positions adjacent  to the current one (in the graph $E$).
Whenever the game position reaches $Y$ the game stops and Player I earns a terminal payoff given by a function $F:Y\to\mathbb{R}$, known to both players beforehand (notice that Player II's earnings are given by $-F$). 
Although we will not consider it here, it is also possible to include a running payoff, i.e., payments that Player I receives from Player II at each intermediate state of the game (see \cite{Peres2009}).

The sets $X$ and $Y$ can be general metric spaces as in \cite{Peres2009}, however there are two particular cases of special interest: the case when $X$ is a graph (in this case $Y\subset X$ and $E=X$), and when $X=\Omega\subset\mathbb{R}^n$.

In \cite{Peres2009} it was proved under very general assumptions that the classical random Tug-of-War game has a value, that is, a function $u(x)$ which represents the expected outcome of the game just described, starting at a point $x\in X$, when both players play optimally. Moreover, the game value satisfies a functional equation known as Dynamic Programming Principle (DPP) 
\[
u(x)=\frac{1}{2}\left(\sup_{y\in\{x\}'}u(y)+\inf_{y\in\{x\}'}u(y)\right)\quad\textnormal{for all}\ \ x\in X,
\]
where $\{x\}'$ denotes the set of neighbors in $E$ of $x\in X$.

In the particular case of $X=\Omega\subset\mathbb{R}^n$, the players are allowed to move to any point in $\overline\Omega$ within a distance $\epsilon$ from $x$. The step size $\epsilon$ is known to both players beforehand.
In this case the DPP reads
\begin{equation}\label{DPP.classic.ToW.intro}
u_\epsilon(x)=\frac{1}{2}\left(\sup_{y\in\overline{B}_\epsilon(x)\cap\overline{\Omega}}u_\epsilon(y)+\inf_{y\in\overline{B}_\epsilon(x)\cap\overline{\Omega}}u_\epsilon(y)\right)\quad\textnormal{for all}\ \ x\in\Omega.
\end{equation}

The key observation in \cite{Peres2009} is that the DPP can be seen as a ``discretization" of the normalized infinity Laplacian. In other words, the limit $u=\lim_{\epsilon\to0}u_\epsilon$ (known as the continuous value of the game in the terminology of \cite{Peres2009}) is a viscosity solution of the Dirichlet problem for the normalized infinity Laplacian, that is,
\begin{equation*}
\begin{cases}
-\Delta_\infty^N u(x)=0, & x\in\Omega;\\
u(x)=F(x), & x\in\partial\Omega.
\end{cases}
\end{equation*}
This is reminiscent of how the Dirichlet problem for the Laplace equation has a probabilistic interpretation in terms of the Brownian motion. The main difference is that all directions are equally probable for the Brownian motion, while the random Tug-of-War considers only the directions of maximal and minimal growth.

\medskip

\subsection{An overview of the infinity Laplacian}
The infinity Laplace operator 
\[
\Delta_\infty u(x)=
\left<D^2u(x)\, \nabla u(x),\nabla u(x)\right>
\]
and its normalized version \eqref{infinityLaplacian_normalized}
appear naturally in
optimal transportation and image processing (see, e.g., \cite{Gangbo1999,Azorero2007}), as well as absolutely minimizing Lipschitz extensions of a given Lipschitz function (see \cite[Section 3]{Rossi2010}.
The interested reader can also check the survey \cite{Lindqvist2014} for a more comprehensive review of the applications of the infinity Laplacian.

 A function $u$ is infinity harmonic if and only if $-\Delta_\infty u=0$ in the viscosity sense, which is equivalent to
  $-\Delta_\infty^Nu=0$ in the viscosity sense. Observe, however, that the infinity Laplacian and the normalized infinity Laplacian are not interchangeable for a general right-hand side $f\not\equiv0$.

Let us just mention that an equivalent characterization of infinity harmonic functions can be based on the following asymptotic mean value formula
\begin{equation}\label{(16)[Lindqvist]}
u(x)=\frac{1}{2}\left(\sup_{y\in\overline{B}_\epsilon(x)}u(y)+\inf_{y\in\overline{B}(x)}u(y)\right)
+o(\epsilon^2)\quad\textnormal{as}\ \ \epsilon\to0
 \end{equation}
(see \cite[Theorem 5.3]{Lindqvist2014}).
In fact, a function $u$ is infinity harmonic in $\Omega$ if and only if $u\in C(\Omega)$ and the mean value formula (\ref{(16)[Lindqvist]}) holds in $\Omega$ in the viscosity sense.

\begin{remark}
It is worth comparing \eqref{DPP.classic.ToW.intro} and \eqref{(16)[Lindqvist]}. Functions satisfying \eqref{DPP.classic.ToW.intro} are called harmonious functions in \cite{LeGruyer-Archer} and are values of classical random Tug-of-War games. As mentioned before, they approximate solutions to the $\infty$-Laplace equation as $\epsilon\to0$ (see \cite{Peres2009}), which satisfy (\ref{(16)[Lindqvist]}) in the viscosity sense (see \cite[Theorem 5.3]{Lindqvist2014}).
\end{remark}

The natural framework to study the infinity Laplacian is the framework of viscosity solutions; it turns out that  one can prescribe smooth boundary values that no $C^2$ solution of $-\Delta_\infty u=0$ can attain. This was proved by Aronsson \cite{Aronsson1967} in the two-dimensional case and by Yu \cite{Yu2006} in higher dimensions. Moreover,  the fact that the operator is not in divergence form, does not allow us to  integrate by parts and define a notion of weak solution.

Moreover, a classical solution of $-\Delta_\infty u=0$ is a viscosity solution but the converse is not true in general. An important example is the function 
\begin{equation*}
u(x,y)=x^{4/3}-y^{4/3},
\end{equation*}
which is infinity harmonic in the viscosity sense but not in the classical one. In fact, this particular function has regularity $C^{1,1/3}$ (see \cite{Crandall2008} for more details).

The regularity of infinity harmonic functions turns out to be a very tough question (see comments on \cite[Section 3]{Evans2007}, \cite[Section 6]{Urbano1997}, \cite[Section 1.5 and 1.8]{Wang2008}). According to \cite{Evans2011}, infinity harmonic functions are differentiable everywhere, while
$C^1$ or $C^{1,\alpha}$ regularity are known to hold in dimension two after the breakthroughs of \cite{Evans2008} and \cite{Savin2005}. It remains an open problem to prove $C^1$ or $C^{1,\alpha}$ regularity in general dimensions.

As the nomenclature ``infinity Laplacian'' suggests, the infinity Laplace equation, $-\Delta_\infty u=0$, is a limit as $p\to\infty$ of the \emph{$p$-Laplace equation}, $-\Delta_p u=0$ (see for instance \cite{Bhattacharya1991}). 
However, the case of the so-called ``infinity Poisson equation" $-\Delta_{\infty}u(x)=f(x)$ is more complex since, in general, it is not the limit as $p\to\infty$ of the corresponding \emph{$p$-Poisson equation} $-\Delta_p u=f$. For instance, in the case $f\equiv1$ the correct limit equation turns out to be
\begin{equation}\label{eq.Jensen.min.intro}
\min\left\{|\nabla u(x)|-1,-\Delta_\infty u(x)\right\}=0
\end{equation}
(see \cite{Bhattacharya1991,Kawohl1990, Lindqvist2014}). Equation \eqref{eq.Jensen.min.intro} is a particular case of Jensen's extremal equations
\begin{equation}\label{eq.jensen.min.intro}
 \min\left\{|\nabla u(x)|-\lambda,-\Delta_\infty u(x)\right\}=0
 \end{equation}
 and
\begin{equation}\label{eq.Jensen.max.intro}
 \max\left\{\lambda-|\nabla u(x)|,-\Delta_\infty u(x)\right\}=0,
 \end{equation}
for $\lambda>0$,
known to have a key role in Jensen's celebrated proof of uniqueness of infinity harmonic functions (see \cite{jensen1993}). 
Notice that for every $\lambda>0$ we have 
 \[
 \min\left\{|\nabla u|-\lambda,-\Delta_\infty^Nu\right\}=0
\quad
\iff
\quad
 \min\left\{|\nabla u|-\lambda,-\Delta_\infty u\right\}=0
 \]
in the viscosity sense, and similarly for \eqref{eq.Jensen.max.intro}.

\subsection{The Totalitarian random Tug-of-War}\label{introduce.Tot.ToW}
We introduce a new game, which we call \emph{Totalitarian random Tug-of-War}, or simply \emph{Totalitarian Tug-of-War}
that is related to Jensen's extremal equations \eqref{eq.jensen.min.intro} and \eqref{eq.Jensen.max.intro}. Even if we will focus our attention on the game played on a graph, the relation between Jensen's extremal equations and the Totalitarian Tug-of-War played in $\Omega\subset\mathbb{R}^n$ will be clarified in Section \ref{limit.eqs.jen}. For simplicity let us focus on equation \eqref{eq.Jensen.min.intro}, since
the general case can be then obtained through simple modifications  that we will describe below.

 The Totalitarian Tug-of-War is a variant of a classical random Tug-of-War game in which one of the players is given extra options. More precisely, the game is played by turns starting with a token placed at a node $x^0$. At each turn, Player I has the power to decide whether to play a round of classical random Tug-of-War (that is, they toss a coin and the winner decides the new game position among the neighbors of the current one), or force Player II choose the new game position among the neighbors of the current one in exchange of a fixed payoff of value $\epsilon$. The fact that Player I somehow imposes at each turn the type of game that is played, is the reason why we refer to this game as Totalitarian Tug-of-War.

The game ends the first time the token reaches a terminal state $x^\tau$, and the payoff Player I receives from Player II is $F_\tau+k_\tau\, \epsilon$ (the game is a two-person, zero-sum game). $F_\tau$ corresponds to the terminal payoff at $x^\tau$ and $k_\tau\in\{0,1,\dots,\tau\}$ is a positive integer that represents the number of times Player I has let Player II decide the next move in exchange of an $\epsilon$-payoff throughout the game. 

Just as in the classical random Tug-of-War game, Player I wants to maximize the final payoff that receives from Player II, who in turn wants to minimize it (and maximize his/her own). In order to attain this objective, the players follow strategies according to which they take a particular action at each turn. Note that all terminal payoffs and the value $\epsilon$ are known to both players beforehand and that players are assumed to play optimally.

When the Totalitarian Tug-of-War is played in $\Omega\subset\mathbb{R}^n$ the payment that Player I receives from
Player II when the latter is forced to choose the new game position is proportional to the 
step size of the classical Tug-of-War game. Denoting the proportionality constant by $\lambda$, we recover \eqref{eq.jensen.min.intro} in the limit as $\epsilon\to0$ (see Section \ref{limit.eqs.jen}). Moreover, considering a Totalitarian Tug-of-War which favors Player II instead of Player I, we can treat \eqref{eq.Jensen.max.intro}. See Section \ref{limit.eqs.jen} for more details.

Let us devote the rest of the introduction to describing our main results in this work along with some key notions that will be used throughout this work.

The first one is the notion of (pure) strategy of a player. A {\it pure strategy} for player $\alpha$, denoted $S_\alpha=\{S_\alpha^k\}_k$, is a sequence of mappings from histories $H_k$ to actions $a^k$. The {\it history} up to stage $k$ is the sequence of game positions and actions up to the $k$-th turn of the game, written as $H^k=(x^0,a^0,x^1,a^1,\dots,a^{k-1},x^k)$, where $x^i$ stands for the game position at the $i$-th turn and $a^i$ means the action carried out by the player who moved in this turn from position $x^i$. 

Roughly speaking, at every turn the strategy indicates the player's next move, provided such player is given the choice, as a function of the current game position and past history. In other words, the strategy of a player says what action to take at each running node of the game, but it depends on the evolution of the game that these actions are accomplished or not. More information about the general notion of strategy can be found in \cite[Section 3.1]{Rossi2010}.

For some games, such as the classical random Tug-of-War and the Totalitarian Tug-of-War,
 the action of a player at a given node is independent of both the stage of the game where the decision is made and the history up to that stage. 

For a game that starts at position $x$, where Players I and II adopt strategies $S_I$ and $S_{II}$ respectively, the expected payoff that Player I receives from Player II is denoted by $\mathbb{E}_{S_I,S_{II}}^{x}[F_\tau+k_\tau\epsilon]$, where $F_\tau$ is the payoff associated to the terminal position $x^\tau$. On the contrary, there may be games for which the players can choose strategies so that the game does not end almost surely. In order to penalize these strategies, the payoff that each player receives from the other is defined to be the worst possible, that is, $-\infty$ in the case of Player I and $+\infty$ for Player II. More precisely, the expected payoff that Player I receives from Player II 
is defined as
\begin{equation}\label{value_of_TT_game_PlayerI}
V_{S_I,S_{II}}^{x}(I):=
\begin{cases}
\mathbb{E}_{S_I,S_{II}}^{x}[F_\tau+k_\tau\, \epsilon],\quad\textnormal{if the game ends almost surely};\\
-\infty,\quad\textnormal{otherwise},
\end{cases}
\end{equation}
where $\mathbb{E}_{S_I,S_{II}}^{x_i}[F_\tau+k_\tau\, \epsilon]$ represents the expected payoff that Player I receives from Player II when they follow strategies $S_I$ and $S_{II}$ respectively and the game starts at the node $x_i$. Analogously, the expected payoff that Player II has to pay Player I is defined as
\begin{equation}\label{value_of_TT_game_PlayerII}
V_{S_I,S_{II}}^{x}(II):=
\begin{cases}
\mathbb{E}_{S_I,S_{II}}^{x}[F_\tau+k_\tau\, \epsilon],\quad\textnormal{if the game ends almost surely};\\
+\infty,\quad\textnormal{otherwise}.
\end{cases}
\end{equation}

Note that $V_{S_I,S_{II}}^{x}(I),V_{S_I,S_{II}}^{x}(II)$ defined in (\ref{value_of_TT_game_PlayerI}) and (\ref{value_of_TT_game_PlayerII}) respectively, can be used to assign a value to the Totalitarian Tug-of-War game in the same way as in the classical random Tug-of-War, the only difference being that in the latter case $k_\tau$ is absent from the definitions of $V_{S_I,S_{II}}^{x}(I),V_{S_I,S_{II}}^{x}(II)$.

More precisely, the \emph{value of the game for Player I} when the game starts at $x$, is defined as
\begin{equation}\label{value_P1_TT}
u^I(x):=\sup_{S_I}\inf_{S_{II}}V_{S_I,S_{II}}^{x}(I),
\end{equation}
and the \emph{value of the game for Player II} when the game starts at $x$, as
\begin{equation}\label{value_P2_TT}
u^{II}(x):=\inf_{S_{II}}\sup_{S_I}V_{S_I,S_{II}}^{x}(II).
\end{equation}
Observe that both $u^I$ and $u^{II}$ represent the worst possible expected payoffs for players I and II respectively, or in other words, $u^I$ and $u^{II}$ are, respectively,
the smallest and largest payoffs that Player I expects to receive from Player II. By definition, $u_I\leq u_{II}$.

Similarly to the classical random Tug-of-War game, we say that the Totalitarian Tug-of-War game has a value when $u^I=u^{II}$.

In the sequel we will refer to the Totalitarian Tug-of-War game played on graphs as the discrete Totalitarian Tug-of-War. 

In Section \ref{Totalitarian ToW has a value} we prove that the discrete Totalitarian Tug-of-War has a value. A key point in the proof of this result is that both $u^I$ and $u^{II}$, as well as the value of the game solve a Dynamic Programming Principle, obtained considering the two possible choices for Player I and then applying conditional probabilities for the coin toss.
In the case of a graph, the aforementioned DPP is, for all interior nodes $x_i$,
\begin{equation}\label{DPP_1}
u_i=\max\left\{\min_{j\in\{i'\}}u_j+\epsilon,\, \frac{1}{2}\left(\max_{j\in\{i'\}}u_j+\min_{j\in\{i'\}}u_j\right)\right\},
\end{equation}
where $u_i=u(x_i)$ and $\{i'\}$ denotes the finite set of indices associated to the nodal neighbors $x_{i'}$ of $x_i$ in the graph.

We prove that the game has a value by means of a discrete comparison principle for the DPP, proved in Section \ref{Totalitarian ToW has a value}.
The precise result is the following.
\begin{theorem}\label{Discrete_Comparison_Principle}
Let $\mathcal{N}$ (nodes) be the finite set of game position indices  and $\mathcal{I}$ the set of indices for the interior (running) nodes of the game. 
 Let $u$ and $v$ be respectively a subsolution and supersolution of
 \[
\min\left\{ u_i-\min_{j\in\{i'\}}u_j-\epsilon,\,  u_i-\frac{1}{2}\left(\max_{j\in\{i'\}}u_j+\min_{j\in\{i'\}}u_j\right)\right\}=0,
\]
i.e., $\mathcal{G}^i[u]\leq0\leq\mathcal{G}^i[v]$ for all $i\in\mathcal{I}$, where 
\[
\mathcal{G}^i[u]:=\min\left\{ u_i-\min_{j\in\{i'\}}u_j-\epsilon,\,  u_i-\frac{1}{2}\left(\max_{j\in\{i'\}}u_j+\min_{j\in\{i'\}}u_j\right)\right\}.
\]
Assume also that $v_i$ is bounded from above for all $i\in\mathcal{N}$ and from below for all $i\in\mathcal{N}\backslash\mathcal{I}$, and $u_i\leq v_i$ for all $i\in\mathcal{N}\backslash\mathcal{I}$. Then $u_i\leq v_i$ for all $i\in\mathcal{N}$.
\end{theorem}

The proof of this discrete comparison principle is based on a change of variables that allows to produce strict supersolutions of the DPP from mere supersolutions and is inspired by the proof of the comparison principle for equation \eqref{eq.Jensen.min.intro} in \cite{Juutinen1998}.

As a consequence, we have the following result.

\begin{theorem}\label{Thm_of_unique_value_1Dgame}
The discrete Totalitarian Tug-of-War game has a value, which is unique.
\end{theorem}

Section \ref{Some_game_examples_on_graphs} contains  explicit examples of Totalitarian Tug-of-War game on star-shaped graphs, with an exhaustive analysis of the graph-segment case. We rely on the fact that the value of the game exists and is the unique solution of the DPP, as proved in Section \ref{Totalitarian ToW has a value}. In this way, we can identify candidates to optimal strategies by direct inspection and it suffices to check that they satisfy the DPP.

Finally, in Section \ref{limit.eqs.jen} we clarify the relation between Jensen's extremal equations and the Totalitarian Tug-of-War played in $\Omega\subset\mathbb{R}^n$.

\medskip

\section{A discrete comparison principle: The Totalitarian Tug-of-War in Graphs has a value}\label{Totalitarian ToW has a value}

We refer to the Totalitarian Tug-of-War game played on graphs as the discrete Totalitarian Tug-of-War. For notational simplicity, we will consider here the case where the graph is a segment; however, the general case follows with the same ideas. We will denote by $\mathcal{N}$ (nodes) the  finite set of game position indices and $\mathcal{I}$ (interior nodes) the set of indices for the running nodes of the game. 

The DPP associated to the discrete Totalitarian Tug-of-War in a discrete graph segment, corresponds to
\begin{equation*}
u_i=\max\left\{\min\left\{u_{i-1},\, u_{i+1}\right\}+\epsilon,\, \frac{1}{2}\left(u_{i-1}+u_{i+1}\right)\right\}\quad\textnormal{for all}\ \ i\in\mathcal{I},
\end{equation*}
where $u_i$ stands for the expected value of the game on the node $x_i$ (compare with (\ref{DPP_1})). Note that this DPP can be equivalently written as 
\begin{equation}\label{1dim_DiscreteDPP}
\min\left\{u_i-\min\left\{u_{i-1},u_{i+1}\right\}-\epsilon,\, u_i-\frac{1}{2}\left(u_{i-1}+u_{i+1}\right)\right\}=0
\end{equation}
for all $i\in\mathcal{I}$. For convenience, we will rewrite (\ref{1dim_DiscreteDPP}) as $\mathcal{G}^i[u]=0$, where $u=(u_0,\dots,u_{n+1})$ and 
\[
\mathcal{G}^i[u]:=\min\left\{u_i-\min\left\{u_{i-1},u_{i+1}\right\}-\epsilon,\, u_i-\frac{1}{2}\left(u_{i-1}+u_{i+1}\right)\right\}
\]
for all $i\in\mathcal{I}$. The combination of the DPP with the terminal payoff, gives us the following Dirichlet problem associated to the discrete Totalitarian Tug-of-War
\begin{equation}\label{discrete_Dirichlet_Problem}
\begin{cases}
\mathcal{G}^i[u]=0,& i\in\mathcal{I};\\
u_i=F_i,& i\in\mathcal{N}\backslash\mathcal{I}.
\end{cases}
\end{equation}

Considering the two possible choices for Player I and then applying conditional probabilities for the coin toss, we obtain the following result. 
\begin{proposition}\label{discrete_auxProp}
The value functions of the game for Players I and II, $u_i^I$ and $u_i^{II}$ respectively, are both solutions to the discrete Dirichlet problem (\ref{discrete_Dirichlet_Problem}). 
\end{proposition}

The importance of Proposition \ref{discrete_auxProp} lies on the fact that it allows us to apply PDE methods to study $u_i^I$ and $u_i^{II}$. In particular, we prove a discrete comparison principle, Theorem \ref{Discrete_Comparison_Principle}, that allows us to prove that (\ref{discrete_Dirichlet_Problem}) has a unique solution and therefore conclude that the game has a value.

The following lemma, is necessary for the proof of Theorem \ref{Discrete_Comparison_Principle}.

\begin{lemma}\label{Discrete_Comparison_Principle_auxLemma}
Let $v$ be a supersolution to $\mathcal{G}^i[v]=0$, bounded from above for all $i\in\mathcal{N}$ and from below for all $i\in\mathcal{N}\backslash\mathcal{I}$. Then, for every $\gamma>0$ there exists a supersolution $\tilde{v}$ to the equation $\mathcal{G}^i[\tilde{v}]=\mu$ for some constant $\mu=\mu(\gamma,v)>0$. Moreover, $\tilde{v}_i-v_i\leq\gamma$ for all $i\in\mathcal{N}$ and $\tilde{v}_i-v_i\geq-\gamma$ for all $i\in\mathcal{N}\backslash\mathcal{I}$.
\end{lemma}

\begin{proof}
We look for $\tilde{v}$ of the form $\tilde{v}_i=g(v_i)$ for all $i\in\mathcal{I}$, where
\begin{equation*}
g(\alpha)=(1+\varepsilon)\, \alpha-\frac{\varepsilon}{4C}\, \alpha^2
\end{equation*}
for $\varepsilon>0$ and $C=\max_{i\in\mathcal{I}}|v_i|$ (recall that $v$ is bounded from above by hypothesis). The constant $\gamma$ in the statement of the lemma will be chosen later as a function of $\varepsilon$.

We assume, without loss of generality, that $\max\{v_i-v_{i-1},\, v_i-v_{i+1}\}=v_i-v_{i-1}$. Note also that $\max\{v_i-v_{i-1},\, v_i-v_{i+1}\}\geq\epsilon$, in view of (\ref{1dim_DiscreteDPP}). According to these,
\begin{align*}
&\max\{\tilde{v}_i-\tilde{v}_{i-1},\, \tilde{v}_i-\tilde{v}_{i+1}\}-\epsilon = \max\{g(v_i)-g(v_{i-1}),\, g(v_i)-g(v_{i+1})\}-\epsilon \\
 &\qquad\qquad= \max\left\{(1+\varepsilon)\, (v_i-v_{i-1})-\frac{\varepsilon}{4C}\, (v_i^2-v_{i-1}^2),\right.\\
 &\qquad\qquad\qquad\qquad\qquad\qquad\qquad\left.(1+\varepsilon)\, (v_i-v_{i+1})-\frac{\varepsilon}{4C}\, (v_i^2-v_{i+1}^2)\right\}-\epsilon\\
 &\qquad\qquad= \max\left\{\left((1+\varepsilon)-\frac{\varepsilon}{4C}\, (v_i+v_{i-1})\right)\, (v_i-v_{i-1}),\right.\\
 &\qquad\qquad\qquad\qquad\qquad\qquad\qquad\left.\left((1+\varepsilon)-\frac{\varepsilon}{4C}\, (v_i+v_{i+1})\right)\, (v_i-v_{i+1})\right\}-\epsilon\\
 &\qquad\qquad \geq\left((1+\varepsilon)-\frac{\varepsilon}{4C}\, 2C\right)\, (v_i-v_{i-1})-\epsilon\\
 &\qquad\qquad =\left(1+\frac{\varepsilon}{2}\right)\, (v_i-v_{i-1})-\epsilon
 \geq\left(1+\frac{\varepsilon}{2}\right)\, \epsilon-\epsilon
 =\frac{\varepsilon}{2}\, \epsilon.
\end{align*}

On the other hand,
\begin{align*}
&(\tilde{v}_i-\tilde{v}_{i-1})+(\tilde{v}_i-\tilde{v}_{i+1})=\left(g(v_i)-g(v_{i-1})\right)+\left(g(v_i)-g(v_{i+1})\right)\\
 &\qquad =(1+\varepsilon)\, \left((v_i-v_{i-1})+(v_i-v_{i+1})\right)-\frac{\varepsilon}{4C}\, \left(v_i^2-v_{i-1}^2\right)-\frac{\varepsilon}{4C}\, \left(v_i^2-v_{i+1}^2\right)\\
 &\qquad =(1+\varepsilon)\, \left(2v_i-v_{i-1}-v_{i+1}\right)+\frac{\varepsilon}{4C}\, \left(-2v_i^2+v_{i-1}^2+v_{i+1}^2\right).
\end{align*}
Note that
\begin{align*}
 &-2v_i^2+v_{i-1}^2+v_{i+1}^2=-2v_i^2+v_{i-1}^2+v_{i+1}^2\pm\left(2v_i^2+2v_i\, v_{i-1}+2v_i\, v_{i+1}\right)\\[0.25em]
 &\qquad\qquad =-4v_i^2+(v_{i-1}-v_i)^2+(v_{i+1}-v_i)^2+2v_i\, v_{i-1}+2v_i\, v_{i+1}\\[0.25em]
 &\qquad\qquad =-2v_i\, (2v_i-v_{i-1}-v_{i+1})+(v_{i-1}-v_i)^2+(v_{i+1}-v_i)^2.
\end{align*}
From these, it follows that
\begin{align*}
&(\tilde{v}_i-\tilde{v}_{i-1})+(\tilde{v}_i-\tilde{v}_{i+1}) =(1+\varepsilon)\, \left(2v_i-v_{i-1}-v_{i+1}\right)\\
 &\hspace{85pt}+\frac{\varepsilon}{4C}\, \left(-2v_i\, (2v_i-v_{i-1}-v_{i+1})+(v_{i-1}-v_i)^2+(v_{i+1}-v_i)^2\right)\\
 &\qquad\qquad =\left(1+\varepsilon-\frac{\varepsilon\, v_i}{2C}\right)\, \left(2v_i-v_{i-1}-v_{i+1}\right)+\frac{\varepsilon}{4C}\, \left((v_{i-1}-v_i)^2+(v_{i+1}-v_i)^2\right)\\
 &\qquad\qquad \geq\left(1+\varepsilon-\frac{\varepsilon\, v_i}{2C}\right)\, \left(2v_i-v_{i-1}-v_{i+1}\right)+\frac{\varepsilon}{4C}\, \left(\max\{v_i-v_{i-1},\, v_i-v_{i+1}\}\right)^2\\
 &\qquad\qquad \geq\frac{\varepsilon}{4C}\, \epsilon^2,
\end{align*}
where in the last inequality we have used that, since $\mathcal{G}^i[v]\geq0$ by hypothesis for all $i\in\mathcal{I}$, $2v_i-v_{i-1}-v_{i+1}\geq0$ and $\max\{v_i-v_{i-1},\, v_i-v_{i+1}\}\geq\epsilon$.

Then,  we get that
\[
\begin{split}
&\min\left\{\tilde{v}_i+\max\left\{-\tilde{v}_{i-1},-\tilde{v}_{i+1}\right\}-\epsilon,\, \tilde{v}_i-\frac{1}{2}\left(\tilde{v}_{i-1}+\tilde{v}_{i+1}\right)\right\} \\
&\hspace{140pt}\geq \min\left\{\frac{\varepsilon}{2}\, \epsilon,\frac{\varepsilon}{8C}\, \epsilon^2\right\}
	= \frac{\varepsilon}{2}\, \epsilon\, \min\left\{1,\frac{\epsilon}{4C}\right\}=\mu
\end{split}
\]
for all $i\in\mathcal{I}$, where $\mu=\mu(\varepsilon,v)>0$.

About the second part of this lemma, since $g(\alpha)-\alpha\leq\frac{3}{4}\, \varepsilon\, C$ for $\alpha\leq C$ and because $\tilde{v}_i=g(v_i)$, it follows that $\tilde{v}_i-v_i\leq\frac{3}{4}\, \varepsilon\, C$ for all $i\in\mathcal{N}$. Similarly, since $g(\alpha)-\alpha\geq-\varepsilon\, D\left(1+\frac{D}{4C}\right)$ for $\alpha\geq -D=-\left|\min_{i\in\mathcal{N}\backslash\mathcal{I}}v_i\right|$, it follows that $\tilde{v}_i-v_i\geq-\varepsilon\, D\left(1+\frac{D}{4C}\right)$ for all $i\in\mathcal{N}\backslash\mathcal{I}$. The result holds taking
\begin{equation*}
\gamma=\varepsilon\max\left\{\frac{3}{4}\, C,\, D\left(1+\frac{D}{4C}\right)\right\}>0,
\end{equation*}
provided $\varepsilon$ is small enough.
\end{proof}

We proceed now with the proof of the discrete comparison principle, Theorem \ref{Discrete_Comparison_Principle},  inspired by the proof  of the comparison principle for equation \eqref{eq.Jensen.min.intro} in \cite{Juutinen1998}.

\begin{proof}[Proof of Theorem \ref{Discrete_Comparison_Principle}]
Arguing by contradiction, we suppose that $\max_{i\in\mathcal{N}}(u_i-v_i)>0$. Since $u_i\leq v_i$ for all $i\in\mathcal{N}\backslash\mathcal{I}$, it follows that there is an index $j\in\mathcal{I}$ such that $u_j-v_j=\max_{i\in\mathcal{N}}(u_i-v_i)>0$. On the other hand, by Lemma \ref{Discrete_Comparison_Principle_auxLemma}, for every $\gamma>0$ there exists $\tilde{v}$ such that $\tilde{v}_i-v_i\leq\gamma$ for all $i\in\mathcal{N}$. As a result, $u_j-v_j>\gamma\geq\tilde{v}_j-v_j$ for $\gamma$ small enough and therefore $u_j>\tilde{v}_j$.

This implies that there is an index $k\in\mathcal{N}$ such that $u_k-\tilde{v}_k=\max_{i\in\mathcal{N}}(u_k-\tilde{v}_k)>0$. In fact, $k\in\mathcal{I}$ since by Lemma \ref{Discrete_Comparison_Principle_auxLemma}, we can assume $\tilde{v}_i-v_i\geq-\gamma$ for all $i\in\mathcal{N}\backslash\mathcal{I}$ and therefore $u_k-\tilde{v}_k>\gamma\geq v_i-\tilde{v}_i\geq u_i-\tilde{v}_i$ for all $i\in\mathcal{N}\backslash\mathcal{I}$. For the sake of simplicity let us assume this index $k$ to be the same $j$ as before. 

It follows by definition that
\begin{equation*}
u_j-\tilde{v}_j\geq u_i-\tilde{v}_i\quad\textnormal{for all}\ \ i\in\mathcal{N}.
\end{equation*}
In particular $u_j-u_{j-1}\geq\tilde{v}_j-\tilde{v}_{j-1}$ and $u_j-u_{j+1}\geq\tilde{v}_j-\tilde{v}_{j+1}$. According to this and writing (\ref{1dim_DiscreteDPP}) as
\begin{equation*}
\min\left\{\max\left\{u_i-u_{i-1},\, u_i-u_{i+1}\right\}-\epsilon,\, (u_i-u_{i-1})+(u_i-u_{i+1})\right\}=0,
\end{equation*}
we have that $\mathcal{G}^j[\tilde{v}]\leq\mathcal{G}^j[u]$, a contradiction with the  fact that $\mathcal{G}^i[u]\leq0<\mathcal{G}^i[\tilde{v}]$ for all $i\in\mathcal{I}$ by Lemma \ref{Discrete_Comparison_Principle_auxLemma}.
\end{proof}

%
%

An important consequence of Theorem \ref{Discrete_Comparison_Principle} is the following bound.

\begin{corollary}\label{u_i_bdd}
Let $u_i$ be a solution to the Dirichlet problem (\ref{discrete_Dirichlet_Problem}) and let $F_i$ be bounded for all $i\in\mathcal{N}\backslash\mathcal{I}$. Then, for all $i\in\mathcal{N}$ and $K=\max_{i\in\mathcal{N}\backslash\mathcal{I}}|F_i|$,
\begin{equation*}
\epsilon\min\{(n+1)-i,\, i\}-K\leq u_i\leq\epsilon\min\{(n+1)-i,\, i\}+K.
\end{equation*}
\end{corollary}

\begin{proof}
Consider, for $F_i$ the same in (\ref{discrete_Dirichlet_Problem}), the Dirichlet problem
\begin{equation*}
\begin{cases}
\mathcal{G}^i[v]=0,& i\in\mathcal{I};\\
v_i=K,& i\in\mathcal{N}\backslash\mathcal{I},
\end{cases}
\end{equation*}
a solution of which is $v_i=\epsilon\min\{(n+1)-i,\, i\}+K$. Then, since it is bounded for all $i\in\mathcal{N}$, by Theorem \ref{Discrete_Comparison_Principle} it follows that $u_i\leq v_i$ for all $i\in\mathcal{N}$ and therefore the upper bound for $u_i$ is proved.

About the lower bound of $u_i$, consider now the Dirichlet problem 
\begin{equation*}
\begin{cases}
\mathcal{G}^i[v]=0,& i\in\mathcal{I};\\
v_i=-K,& i\in\mathcal{N}\backslash\mathcal{I},
\end{cases}
\end{equation*}
a solution of which is $v_i=\epsilon\min\{(n+1)-i,\, i\}-K$. On the other hand, since $u_i$ is a solution to (\ref{discrete_Dirichlet_Problem}) by hypothesis, it is in particular a supersolution and bounded from above due to the first part of the proof. Then, by Theorem \ref{Discrete_Comparison_Principle} it follows that $v_i\leq u_i$ for all $i\in\mathcal{N}$ and the proof is finished.
\end{proof}

In order to prove that the discrete Totalitarian Tug-of-War game has a value, i.e., $u_i^I=u_i^{II}$ for all $i\in\mathcal{N}$, it will be enough to show that 
$u_i^{II}\leq u_i^I$  holds for all $i\in\mathcal{N}$, since  the converse holds by definition.

\begin{proof}[Proof of Theorem \ref{Thm_of_unique_value_1Dgame}]
Let $u_i^I$, $u_i^{II}$ be the respective value of the Totalitarian Tug-of-War game for Players I and II, for $i\in\mathcal{N}$. By Proposition \ref{discrete_auxProp}, both $u_i^I$ and $u_i^{II}$ are solutions to the discrete Dirichlet problem (\ref{discrete_Dirichlet_Problem}) and hence, according to Corollary \ref{u_i_bdd} they are bounded for all $i\in\mathcal{N}$. 
In particular, they are respectively a supersolution and subsolution of (\ref{1dim_DiscreteDPP}) for all $i\in\mathcal{N}$ and $u_i^{II}\leq u_i^I$ for all $i\in\mathcal{N}\backslash\mathcal{I}$. Thus, we can apply Theorem \ref{Discrete_Comparison_Principle} for $u_i=u_i^{II}$ and $v_i=u_i^I$, so that $u_i^{II}\leq u_i^I$ for all $i\in\mathcal{N}$. On the other hand, by definition, $u_i^I\leq u_i^{II}$ for all $i\in\mathcal{N}$. It then follows that $u_i^{II}=u_i^I$ for all $i\in\mathcal{N}$ so that the discrete Totalitarian Tug-of-War game has a unique value.
\end{proof}

It is worth mentioning here that a more exhaustive discussion of discretized degenerate elliptic equations and its applications to numerical analysis can be found in \cite{Oberman2006}, for instance.


\section{Examples}\label{Some_game_examples_on_graphs}

This section is devoted to presenting explicit examples of the Totalitarian Tug-of-War on graphs. Here, we rely on the fact that the value of the game exists and is the unique solution of the DPP, as we proved in Section \ref{Totalitarian ToW has a value}. In this way, we can identify candidates to optimal strategies by direct inspection and it suffices to check if the corresponding expected value satisfies the DPP.

\subsection{Game on a graph segment}\label{One-dimensional game with multiple running nodes}

In this section we analyze in detail the Totalitarian Tug-of-War game played on a graph segment with $n$ running nodes, see Figure \ref{1Dgame_MultN}. Recall that the values of $\epsilon$, $F_0$ and $F_{n+1}$ are fixed and known to both players beforehand, and the final payoffs are independent of the $\epsilon$-payments that took place throughout the game. Notice that we can assume $F_0\leq F_{n+1}$ due to the symmetry of the graph.

\begin{figure}
\centering
\begin{tikzpicture}[scale=.9]
	\draw[semithick] (0,0) -- (1.5,0);
 \draw[fill] (0,0) circle [radius=5pt];
 \draw (0,0.75) node {$x_0$};
 \draw[fill] (1.5,0) circle [radius=3pt];
 \draw (1.5,0.75) node {$x_1$};
 \draw[semithick] (1.5,0) -- (3,0);
 \draw[fill] (3,0) circle [radius=3pt];
 \draw (3,0.75) node {$x_2$}; 
 \draw[semithick] (3,0) -- (3.5,0);
 \draw[loosely dotted, semithick] (4,0) -- (6,0);
 \draw[semithick] (6.5,0) -- (7,0);
 \draw[fill] (7,0) circle [radius=3pt];
 \draw (7,0.75) node {$x_{n-1}$};
 \draw[semithick] (7,0) -- (8.5,0);
 \draw[fill] (8.5,0) circle [radius=3pt];
 \draw (8.5,0.75) node {$x_n$};
 \draw[semithick] (8.5,0) -- (10,0);
 \draw[fill] (10,0) circle [radius=5pt];
 \draw (10,0.75) node {$x_{n+1}$};
\end{tikzpicture}
\caption{Game positions of the game on a segment with multiple running nodes.}
\label{1Dgame_MultN}
\end{figure}
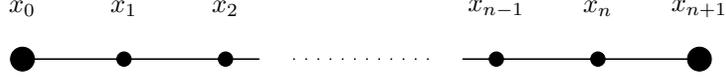

The main result in this section is the following theorem, that describes the value of the game and strategies that realize it for any given $F_{n+1},F_0$ and $\epsilon$.

\begin{theorem}\label{examples.main.thm.segment}
Consider a Totalitarian Tug-of-War with parameter $\epsilon$ played on a graph segment with $n$ running nodes $x_0,x_1,\ldots, x_{n}, x_{n+1}$, $n\geq2$.
Let $F_0$ and $F_{n+1}$ be the terminal payoffs at $x_0$ and $x_{n+1}$ respectively, which we can assume
 $F_0\leq F_{n+1}$. Define the quantity
 \[
Q=\frac{F_{n+1}-F_0}{\epsilon}\geq0.
 \]
 Then, the value of the game and strategies that realize it can be described exhaustively in terms of $Q$ as follows:
\begin{itemize}
 \item[(a)] If $Q> n+1$ then the value of the game is given by
\[
u_i=
\left(\frac{n+1-i}{n+1}\right)F_0+\frac{i}{n+1}\,F_{n+1},\quad\textrm{for}\ 1\leq i\leq n.
\]
In this case the value of the game is realized when the players play a classical random Tug-of-War, that is, Player I's strategy is to always move towards $x_{n+1}$ and Player II's strategy is to always move towards $x_{0}$.

 \item[(b)] If $n-1< Q\leq n+1$ then 
\[
u_i=F_0+i\, \epsilon, \quad \textnormal{for}\ 1\leq i\leq n.
\]
In this case the value of the game is realized when Player I always lets Player II decide the next move, who moves towards the terminal node $x_0$.

Moreover, if $Q=n+1$ there is a family of pairs of strategies that yield the same value of the game. This family includes the pair of strategies just mentioned as a particular case, and it can be described as follows: For each fixed integer $j\in[0,n]$ Player I lets Player II decide the next move at the nodes $x_i$ with $1\leq i\leq j$, who moves towards the terminal node $x_{0}$, and
Players I and II play a classical random Tug-of-War at the nodes $x_i$ with $j+1\leq i\leq n$ (Players I and II move towards $x_{n+1}$ and $x_0$ respectively).

 \item[(c)] If there exists an integer $k\in[1,n-1]$ such that
 $2k-n-1< Q\leq 2k-n+1,$
then 
\[
u_i=
\left\{
\begin{split}
&F_0+i\, \epsilon, \hspace{70pt} \textnormal{for}\ 1\leq i\leq k;\\ 
&F_{n+1}+(n+1-i)\, \epsilon, \quad \textnormal{for}\ k+1\leq i\leq n.
\end{split}
\right.
\]

In this case the value of the game is realized when Player I always lets Player II decide the next move, who moves towards the terminal node $x_0$ at the nodes $x_i$ with $1\leq i\leq k$, and towards $x_{n+1}$ at the nodes $x_i$ with $k+1\leq i\leq n$.


Moreover, if $Q=2k-n+1$ there is a family of pairs of strategies that yield the same value of the game. This family includes the pair of strategies just mentioned as a particular case, and it can be described as follows: For each fixed integer $j\in[0,n-2]$,
\begin{itemize}
\item Player I lets Player II decide the next move at the nodes $x_i$ with $1\leq i\leq j$, who moves towards the terminal node $x_{0}$.
\item Players I and II play a classical random Tug-of-War at the nodes $x_i$ with $j+1\leq i\leq k\leq n-1$ (Players I and II move towards $x_{n+1}$ and $x_0$ respectively).
\item Player I lets Player II decide the next move at the nodes $x_i$ with $k+1\leq i\leq n$, who moves towards the terminal node $x_{n+1}$.
\end{itemize}

\end{itemize}
 
\end{theorem}

\begin{remark}
 Note that in the case $Q>0$, as $\epsilon\to0$ we end up on the first situation, so the limit is, naturally, a classical random Tug-of-War. 
\end{remark}

\begin{remark}
The case $Q=0$, that is, when $F_0=F_{n+1}$, corresponds to case (c) taking $k=(n-1)/2$ or $k=n/2$, depending on the parity of $n$.
\end{remark}

Let us start the proof of Theorem \ref{examples.main.thm.segment} by examining heuristically the choices available to both players that motivate the analysis below. On the one hand, at any given running node, Player I has three options: to move towards $x_0$, towards $x_{n+1}$, or to let Player II decide. Moving towards $x_0$ is not reasonable, since $F_0\leq F_{n+1}$ implies that the new game position chosen by Player I would actually be more favorable to Player II. In fact, Player I would reach the same position and receive an $\epsilon$-payoff by having Player II decide next move. 

On the other hand, Player II has two options: to move towards $x_0$ or towards $x_{n+1}$. In principle, moving towards $x_{n+1}$ (the terminal node with largest terminal payoff) seems against Player II's interest. However, there can be situations where it may be preferable for Player II to minimize the damage and avoid $\epsilon$-payments by ending the game at $x_{n+1}$ as quickly as possible.

This heuristic reasoning yields three regions in the segment according to both player's choices, which can be described in terms of two indices $j_1\leq n$, $j_2\geq 1$ such that $j_1<j_2$ as follows: 
\begin{enumerate}
\item At the nodes $x_i$, for $0\leq i\leq j_1$, Player I allows Player II to move, who moves towards the terminal node $x_0$;
\item At the nodes $x_i$, for $j_1+1\leq i\leq j_2-1$, Players I and II play classical random Tug-of-War, where the players want to move towards the terminal nodes $x_{n+1}$ and $x_0$, respectively;
\item At the nodes $x_i$, for $j_2\leq i\leq n+1$, Player I allows Player II to move, who moves towards the terminal node $x_{n+1}$.
\end{enumerate}

In fact, we are going to see that the expected value of the game at each position can be completely determined by looking only at pairs of strategies of the form just described, since these include all the reasonable ones. ``Unreasonable" strategies include playing classical random Tug-of-War with the players switching their roles, or playing classical random Tug-of-War when both players want to move towards the same terminal node, in which case Player I would do better by changing strategy and forcing Player II to choose the next move.

One could argue if Player I would be able to get a higher expected payoff 
 by choosing a strategy other than allowing Player II to decide the next move in \emph{every} node $0\leq i\leq j_1$ and $j_2\leq i\leq n+1$.
The answer is negative, and this is the content of the following extremal results, Lemmas \ref{Lemma_reductionOfStrategies_smallNodes} and \ref{Lemma_reductionOfStrategies_largeNodes}.

In Lemma \ref{Lemma_reductionOfStrategies_smallNodes} we prove that, as long as Player II's strategy is to move towards $x_0$ at every $x_i$ with $1\leq i\leq j_1$ and 
Player I's strategy implies letting Player II decide the next move at $x_{j_1}$, the expected value of the game at every node $x_i$ with $1\leq i\leq j_1-1$
is independent of the action chosen by Player I at those game positions.

\begin{lemma}\label{Lemma_reductionOfStrategies_smallNodes}
Let $1\leq j_1\leq n$ and let $S_I,S_{II}$ be any pair of strategies such that Player I allows Player II to move at node $x_{j_1}$, while Player II moves towards the terminal node $x_0$ for all $x_j$ with $1\leq j\leq j_1$. Then, $u_k=F_0+k\,\epsilon$ for all $0\leq k\leq j_1$.
\end{lemma}

\begin{proof}
We argue by strong mathematical induction. Let $P_i$ be the property stated in the lemma, that is:

$P_{i}:$ ``For any pair of strategies $S_I,S_{II}$ such that Player I allows Player II to move at $x_{i}$ while Player II moves towards the terminal node $x_0$ at all $x_j$ with $1\leq j\leq i$, we have that $u_k=F_0+k\,\epsilon$ for all $0\leq k\leq i$".

 Assume $P_{i}$ holds for all $0\leq i\leq j_1-1$. We want to prove that $P_{j_1}$ is also true. There are two possible situations in terms of Player I's strategy at the game positions $x_i$ for $1\leq i\leq j_1-1$.

On the one hand, suppose that Player I allows Player II to move at node $x_k$ for some $1\leq k\leq j_1-1$. Then, since $P_k$ is true by the strong induction hypothesis, the value of the game at the nodes $x_j$ for $0\leq j\leq k$, is $u_j=F_0+j\, \epsilon$.

\begin{figure}
\centering
\begin{tikzpicture}[scale=1]
	\draw[semithick] (0,0) -- (1.5,0);
 \draw[fill] (0,0) circle [radius=5pt];
 \draw (0,1) node {$y_0$};
 \draw (0,0.5) node {\footnotesize{$x_k$}};
 \draw[fill] (1.5,0) circle [radius=3pt];
 \draw (1.5,1) node {$y_1$};
 \draw (1.5,0.5) node {\footnotesize{$x_{k+1}$}};
 \draw[semithick] (1.5,0) -- (2,0);
 \draw[loosely dotted, semithick] (2.5,0) -- (3.5,0);
 \draw[semithick] (4,0) -- (4.5,0);
 \draw[fill] (4.5,0) circle [radius=3pt];
 \draw (4.5,1) node {$y_{j_1-k-1}$};
 \draw (4.5,0.5) node {\footnotesize{$x_{j_1-1}$}};
 \draw[semithick] (4.5,0) -- (6,0);
 \draw[fill] (6,0) circle [radius=3pt];
 \draw (6,1) node {$y_{j_1-k}$};
 \draw (6,0.5) node {\footnotesize{$x_{j_1}$}};
 \draw[semithick] (6,0) -- (6.5,0);
 \draw[loosely dotted, semithick] (7,0) -- (8,0);
 \draw[semithick] (8.5,0) -- (9,0);
 \draw[fill] (9,0) circle [radius=5pt];
 \draw (9,1) node {$y_{n-k+1}$};
 \draw (9,0.5) node {\footnotesize{$x_{n+1}$}};
\end{tikzpicture}
\caption{Game positions after relabeling the original nodes $x_i$ as $y_{i-k}$, for $k\leq i\leq n+1$.}
\label{Lemma_reductionOfStrategies_smallNodes_figure}
\end{figure}
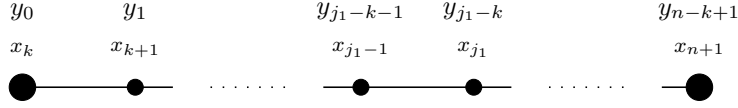

We can relabel the remaining nodes as $y_j=x_{j+k}$, for $0\leq j\leq n+1$ (see Figure \ref{Lemma_reductionOfStrategies_smallNodes_figure}) so that the original node $x_k$ is now a terminal node with associated payoff $F_0+k\, \epsilon$.
By the strong induction hypothesis, $P_{j_1-k}$ holds and therefore
the expected value of the game at $y_j=x_{j+k}$ for $0\leq j\leq j_1-k$, is $u_j=F_0+(k+j)\, \epsilon$, so $P_{j_1}$ holds true.

On the other hand, the opposite situation is also possible, i.e., that Player I does not allow Player II to choose the next move at any node $x_j$ for $1\leq j<j_1$. In this case, Player I will necessarily choose to move towards $x_{n+1}$ (since otherwise it would be more advantageous to let Player II decide next move and collect an $\epsilon$ payoff). Therefore, the players play classical random Tug-of-War in the nodes $x_i$ for $1\leq i\leq j_1-1$ and we can compute $u_{j_1}$ by solving the linear system
\begin{equation*}
\begin{cases}
u_0=F_0;\\
u_i=\frac{1}{2}(u_{i-1}+u_{i+1}),\quad\textnormal{for}\ 1\leq i\leq j_1-1;\\
u_{j_1}=u_{j_1-1}+\epsilon.
\end{cases}
\end{equation*}
Thus, we obtain $u_i=F_0+i\, \epsilon$ for $0\leq i\leq j_1$, which concludes the proof.
\end{proof}

In a similar way, the expected value of the game at every node $x_i$ with $j_2+1\leq i\leq n$
is independent of the action chosen by Player I at those game positions, as long as 
 Player II's strategy is to move towards $x_{n+1}$ at every $x_i$ with $j_2\leq i\leq n$ and 
 Player I's strategy implies letting Player II decide the next move at $x_{j_2}$. This result follows from Lemma \ref{Lemma_reductionOfStrategies_smallNodes} and is collected in the following lemma.

\begin{lemma}\label{Lemma_reductionOfStrategies_largeNodes}
Let $1\leq j_2\leq n$ and let $S_I,S_{II}$ be any pair of strategies such that Player I allows Player II to move at node $x_{j_2}$, while Player II moves towards the terminal node $x_{n+1}$ for all $x_j$ with $j_2\leq j\leq n$. Then, $u_k=F_{n+1}+(n+1-k)\,\epsilon$ for all $j_2\leq k\leq n+1$.
\end{lemma}

With the previous considerations, the computation of the expected payoff for each pair of strategies $S_I,S_{II}$ follows using conditional expectations and Lemmas \ref{Lemma_reductionOfStrategies_smallNodes} and \ref{Lemma_reductionOfStrategies_largeNodes}. 
The expected value of the game $u_i$ at the running node $x_i$ can be written in terms of the indices $j_1$ and $j_2$ as follows
\begin{align}
u_i&=F_0+i\, \epsilon, & \textnormal{for}\ & 0\leq i\leq j_1;\label{1Dgame_MultN_valueLetX0}\\
u_i&=\frac{j_2-i}{j_2-j_1}(F_0+j_1\, \epsilon)+\frac{i-j_1}{j_2-j_1}\big(F_{n+1}+(n+1-j_2)\, \epsilon\big), & \textnormal{for}\ & j_1+1\leq i\leq j_2-1;\label{1Dgame_MultN_valueTow}\\
u_i&=F_{n+1}+(n+1-i)\, \epsilon, & \textnormal{for}\ & j_2\leq i\leq n+1,\label{1Dgame_MultN_valueLetXN}
\end{align}
where \eqref{1Dgame_MultN_valueLetX0} and \eqref{1Dgame_MultN_valueLetXN} follow from Lemmas \ref{Lemma_reductionOfStrategies_smallNodes} and \ref{Lemma_reductionOfStrategies_largeNodes} respectively.

At this point we can use the DPP, which in this case reads
\begin{equation}\label{DPP.examples.segment.n}
u_i=\max\Big\{\min\left\{u_{i-1},u_{i+1}\right\}+\epsilon,\, \frac{1}{2}\Big(\min\left\{u_{i-1},u_{i+1}\right\}+\max\left\{u_{i-1},u_{i+1}\right\}\Big)\Big\},
\end{equation}
 to determine under which conditions on $F_0, F_{n+1}$ and $\epsilon$ do expressions \eqref{1Dgame_MultN_valueLetX0}, \eqref{1Dgame_MultN_valueTow}, and \eqref{1Dgame_MultN_valueLetXN} yield the value of the game.

\begin{remark}\label{example.n.crucial.remark}
It is easy to check that formulas \eqref{1Dgame_MultN_valueLetX0} and \eqref{1Dgame_MultN_valueLetXN} satisfy the DPP \eqref{DPP.examples.segment.n} in the interior nodes of their respective ranges, i.e. for $1\leq i\leq j_1-1$ and $j_2+1\leq i\leq n$ respectively. On the other hand, for every $j_1+1\leq i\leq j_2-1$ , expression \eqref{1Dgame_MultN_valueTow} yields
\[
\begin{split}
&\max\left\{\min\left\{u_{i-1},u_{i+1}\right\}+\epsilon,\, \frac{1}{2}\Big(\min\left\{u_{i-1},u_{i+1}\right\}+\max\left\{u_{i-1},u_{i+1}\right\}\Big)\right\}\\
&
\hspace{+100pt}=
u_i
+
\max\left\{\epsilon-\frac{\left|F_{n+1}-F_0+(n+1-j_1-j_2)\epsilon\right|}{j_2-j_1},0\right\}.
\end{split}
\]
Therefore, \eqref{1Dgame_MultN_valueTow} will solve \eqref{DPP.examples.segment.n} if and only if
\[
\big|F_{n+1}-F_0+(n+1-j_1-j_2)\epsilon\big|\geq\epsilon(j_2-j_1),
\]
which can be rewritten as
\begin{equation}\label{segment.n.first.condition}
\frac{F_{n+1}-F_0}{\epsilon}\leq 2 j_1-(n+1)
\qquad\textrm{or}\qquad
\frac{F_{n+1}-F_0}{\epsilon}\geq 2 j_2-(n+1). 
\end{equation}
\end{remark}

Recall the three regions in the segment described above, i.e., (1), (2), and (3), and notice that there are seven possible game situations resulting from combinations of these. We are going to describe now how
each of these game situations will take place or not depending on 
 the relations among $F_0, F_{n+1}$ and $\epsilon$.
\begin{itemize}
 \item[(i)] Only (1) takes place. In this case $j_1=n$.
 
Observe that the fact that Player II chooses to move from $x_n$ towards $x_0$ and not towards $x_{n+1}$ implies that $F_0+n\epsilon$ (Player I's payoff at $x_n$ in the situation we are considering) is no bigger than $F_{n+1}+\epsilon$ (Player I's payoff at $x_n$ when Player II's strategy at $x_n$ is to move towards $x_{n+1}$). This yields the necessary condition 
\begin{equation}\label{example.n.condition.1.case.1}
n-1\leq \frac{F_{n+1}-F_0}{\epsilon}. 
\end{equation}

Moreover, it follows from Remark \ref{example.n.crucial.remark} that formula \eqref{1Dgame_MultN_valueLetX0} satisfies the DPP \eqref{DPP.examples.segment.n} for $1\leq i\leq j_1-1$, so it only remains to check the conditions under which \eqref{DPP.examples.segment.n} is satisfied at $x_n$.
 In fact, 
 \[
 \min\left\{u_{n-1},u_{n+1}\right\}=\min\big\{F_0+(n-1)\epsilon,F_{n+1}\big\}= F_0+(n-1)\epsilon
 \]
 by condition \eqref{example.n.condition.1.case.1} and then
 \[
\begin{split}
&\max\left\{\min\left\{u_{i-1},u_{i+1}\right\}+\epsilon,\, \frac{1}{2}\Big(\min\left\{u_{i-1},u_{i+1}\right\}+\max\left\{u_{i-1},u_{i+1}\right\}\Big)\right\}\\
&\hspace{90pt}
=
\max\left\{F_0+n\epsilon,\, \frac12\big(F_0+(n-1)\epsilon+F_{n+1}\big)\right\}\\
&\hspace{90pt}
=
u_n+\max\left\{0,\, \frac12\big(F_{n+1}-F_0-(n+1)\epsilon\big)\right\}.
\end{split}
\]

 We conclude that \eqref{1Dgame_MultN_valueLetX0} satisfies \eqref{DPP.examples.segment.n} at $x_n$ if and only if $F_{n+1}-F_0\leq(n+1)\epsilon$, which we can write together with \eqref{example.n.condition.1.case.1} as follows
\begin{equation}\label{example.n.condition.2.case.1}
n-1\leq\frac{F_{n+1}-F_0}{\epsilon}\leq n+1. 
\end{equation}

Observe that this situation corresponds to (b) in the statement of Theorem~\ref{examples.main.thm.segment}.

\medskip

 \item[(ii)] Only (2) takes place, i.e. the Players play a classical random Tug-of-War. This case corresponds to $j_1=0$ and $j_2=n+1$.
 
Notice that according to Remark \ref{example.n.crucial.remark} expression \eqref{1Dgame_MultN_valueTow} will solve \eqref{DPP.examples.segment.n} if and only if
 condition \eqref{segment.n.first.condition} holds, which in this case amounts to
\begin{equation}\label{example.n.condition.1.case.2}
\frac{F_{n+1}-F_0}{\epsilon}\geq n+1. 
 \end{equation}
Observe that when $Q>n+1$ this situation corresponds to (a) in the statement of Theorem \ref{examples.main.thm.segment} and when $Q=n+1$ to the alternative strategy with $j=0$ in (b).

\medskip

 \item[(iii)] Only (3) takes place. In this case $j_2=1$. However, this alternative is not reasonable since Player II would get a better payoff moving towards $x_0$ at $x_1$, a situation that will be covered in the next case.
 
 \medskip

 \item[(iv)] Both (1) and (3) take place but not (2), i.e., Player I always lets Player II decide the next move. This case corresponds to $j_2=j_1+1\in[2,n]$.
 
 Similarly to the reasoning in (i), observe that the fact that Player II chooses to move from $x_{j_1}$ towards $x_0$ and not towards $x_{n+1}$ implies that $F_0+j_1\epsilon$ (Player I's payoff at $x_{j_1}$ in the situation we are considering) is no bigger than $F_{n+1}+(n-j_1+1)\epsilon$ (Player I's payoff at $x_{j_1}$ when Player II's strategy at $x_{j_1}$ is to move towards $x_{n+1}$). This yields the necessary condition 
\begin{equation}\label{example.n.condition.1.case.4}
2 j_1-n-1 \leq \frac{F_{n+1}-F_0}{\epsilon}.
\end{equation}

Analogously, the fact that Player II chooses to move from $x_{j_1+1}$ towards $x_{n+1}$ and not towards $x_0$ implies that 
$F_{n+1}+(n-j_1)\epsilon$
(Player I's payoff at $x_{j_1+1}$ in the situation we are considering) is no bigger than 
$F_0+(j_1+1)\epsilon$ 
(Player I's payoff at $x_{j_1+1}$ when Player II's strategy at $x_{j_1+1}$ is to move towards $x_0$). This yields the necessary condition 
\[
\frac{F_{n+1}-F_0}{\epsilon}
\leq
2j_1-n+1,
\]
which we can write altogether with \eqref{example.n.condition.1.case.4} as
\begin{equation}\label{example.n.condition.3.case.4}
2 j_1-n-1
\leq 
\frac{F_{n+1}-F_0}{\epsilon}
\leq
2j_1-n+1.
\end{equation}

Furthermore, it follows from Remark \ref{example.n.crucial.remark} that 
\begin{equation*}
u_i=
\left\{
\begin{split}
&F_0+i\, \epsilon, \hspace{70pt} \textnormal{for}\ 1\leq i\leq j_1;\\ 
&F_{n+1}+(n+1-i)\, \epsilon, \quad \textnormal{for}\ j_1+1\leq i\leq n,
\end{split}
\right.
\end{equation*}
 satisfies the DPP \eqref{DPP.examples.segment.n} for every $ i\neq j_1,j_1+1$, so it only remains to check \eqref{DPP.examples.segment.n} in those cases. We provide the details in the case $j_1$, since the case $j_1+1$ follows similarly.
 In fact, condition \eqref{example.n.condition.3.case.4} implies
$ \min\left\{u_{j_1-1},u_{j_1+1}\right\}=u_{j_1-1}$
 and 
 $ \max\left\{u_{j_1-1},u_{j_1+1}\right\}=u_{j_1+1}$
and then
 \[
\begin{split}
&\max\left\{\min\left\{u_{j_1-1},u_{j_1+1}\right\}+\epsilon,\, \frac{1}{2}\Big(\min\left\{u_{j_1-1},u_{j_1+1}\right\}+\max\left\{u_{j_1-1},u_{j_1+1}\right\}\Big)\right\}\\
&\hspace{90pt}
=
u_{j_1}+\max\left\{0,\, \frac12\big(F_{n+1}-F_0-(2j_1-n+1)\epsilon\big)\right\}=u_{j_1},
\end{split}
\]
again by condition \eqref{example.n.condition.3.case.4}.

 Observe that this situation corresponds to (c) in the statement of Theorem \ref{examples.main.thm.segment} with $k=j_1$.

\medskip

 \item[(v)] Both (1) and (2) take place but not (3). In this case $j_1\in[1,n-1]$ and $j_2=n+1$.
 
 We can think about this situation as a classical random Tug-of-War game played in a shorter segment with nodes $x_{j_1},x_{j_1+1},\ldots, x_{n+1}$ and final payoffs $F_0+j_{1}\epsilon$ at $x_{j_1}$ and $F_{n+1}$ at $x_{n+1}$. Then, we recall from (ii) that condition \eqref{example.n.condition.1.case.2} must hold, which in this case reads,
\[
\frac{F_{n+1}-(F_0+j_{1}\epsilon)}{\epsilon}\geq n-j_{1}+1,
 \]
 which amounts to \eqref{example.n.condition.1.case.2} after simplification.

On the other hand, we can also reduce this situation to case (i) in a shorter segment with nodes $x_0,x_1,\ldots,x_{j_1},x_{j_1+1}$ and final payoffs $F_0$ at $x_{0}$ and
\[
\left(\frac{n-j_1}{n+1-j_1}\right)\big(F_0+j_1\, \epsilon\big)+\frac{F_{n+1}}{n+1-j_1}
\]
 at $x_{j_1+1}$, see formula \eqref{1Dgame_MultN_valueTow}. We conclude from condition \eqref{example.n.condition.2.case.1} that
\[
(j_1-1)\,\epsilon\leq\frac{(n-j_1)(F_0+j_1\, \epsilon)+F_{n+1}}{n+1-j_1}-F_0\leq (j_1+1)\,\epsilon,
\]
which reduces to
\[
2j_1-n-1\leq \frac{F_{n+1}-F_0}{\epsilon}\leq n+1,
\]
and
in view of \eqref{example.n.condition.1.case.2} yields
\begin{equation}\label{example.n.condition.2.case.5}
F_{n+1}=F_0+(n+1)\epsilon.
\end{equation}
Moreover, it can be easily checked that with this condition $u_i= F_0+i\, \epsilon$ for all $1\leq i\leq n$,
which verifies the DPP \eqref{DPP.examples.segment.n}.

This situation corresponds to (b) in the statement of Theorem \ref{examples.main.thm.segment}, choosing $j=j_1$.

\medskip

 \item[(vi)] Both (2) and (3) take place but not (1). In this case $j_1=0$ and $j_2\leq n$.

Arguing as in the previous cases, the fact that Player II chooses to move from $x_{j_2}$ towards $x_{n+1}$ and not towards $x_{0}$ implies that Player I's expected payoff at $x_{j_2}$ in this situation, i.e. $F_{n+1}+(n+1-j_2)\, \epsilon$
 is no bigger than Player I's expected payoff at $x_{j_2}$ when Player II's strategy at $x_{j_2}$ is to move towards $x_{0}$, that is,
\[
\epsilon+\frac{F_0}{j_2}+\frac{j_2-1}{j_2}\big(F_{n+1}+(n+1-j_2)\, \epsilon\big).
\]
 This relation yields the condition 
\begin{equation}\label{example.n.condition.1.case.6}
\frac{F_{n+1}-F_0}{\epsilon}\leq 2j_2-n-1.
\end{equation}

 On the other hand, we can think about this situation as a classical random Tug-of-War game played in a shorter segment with nodes $x_0,x_1,\ldots, x_{j_2}$ and final payoffs $F_0$ at $x_0$ and $F_{n+1}+(n+1-j_2)\, \epsilon$ at $x_{j_2}$. Then, we recall from (ii) that condition \eqref{example.n.condition.1.case.2} must hold, which in this case reads,
 \[
\frac{F_{n+1}-F_0}{\epsilon}\geq 2j_2-n-1.
 \]
 This condition, altogether with \eqref{example.n.condition.1.case.6} yields the necessary condition
\begin{equation}\label{condic.dim.1.n.cases.6.7}
F_{n+1}=F_0 + (2j_2-n-1)\,\epsilon.
\end{equation}
Then, it can be easily checked that 
\begin{equation}\label{formula.dim.1.n.cases.6.7}
u_i=
\left\{
\begin{split}
&F_0+i\, \epsilon, \hspace{70pt} \textnormal{for}\ 1\leq i\leq j_2-1;\\ 
&F_{n+1}+(n+1-i)\, \epsilon, \quad \textnormal{for}\ j_2\leq i\leq n,
\end{split}
\right.
\end{equation}
which verifies the DPP \eqref{DPP.examples.segment.n}.

Notice that this situation corresponds to (c) in the statement of Theorem \ref{examples.main.thm.segment} with $j=j_1=0$ and $k=j_2-1$.

\medskip

 \item[(vii)] All three ranges (1), (2), and (3) take place. In this case $1\leq j_1\leq j_2-2<j_2\leq n$.

We can reduce this situation to case (v) in a shorter segment with nodes $x_0,x_1,\ldots, x_{j_2}$ and final payoffs $F_0$ at $x_0$ and $F_{n+1}+(n+1-j_2)\, \epsilon$ at $x_{j_2}$.
Then, \eqref{example.n.condition.2.case.5} yields condition \eqref{condic.dim.1.n.cases.6.7}
and $u_i$ is given by \eqref{formula.dim.1.n.cases.6.7},
which verifies the DPP \eqref{DPP.examples.segment.n}.

This situation corresponds to (c) in the statement of Theorem \ref{examples.main.thm.segment}, where we have chosen $j=j_1$ and $k=j_2-1$.

\end{itemize}


This concludes the proof of Theorem \ref{examples.main.thm.segment}.

\subsection{Star-shaped graphs}\label{example.graph.Y.general}

Now we study the Totalitarian Tug-of-War game in a star-shaped graph consisting of an arbitrary number of (arbitrarily long) graph segments glued together at a common endpoint. The simplest of these is the $Y$ graph resulting from gluing together three graph segments, see Figure \ref{Ygame_MultN}. For simplicity, we will denote by $x_{*}$ the node where all branches meet and by $x_{n_i+1}$ the terminal node of each branch $i=1,2,\ldots k$.

The key observation is that the game on a star-shaped graph 
can be reduced to a game on a graph segment  where we can apply Theorem \ref{examples.main.thm.segment}. 
Furthermore, one can obtain in at most $k(k-1)/2$ steps  the game value and strategies that realize it for any given star-shaped graph as follows:

\begin{enumerate}
 \item We choose  two terminal nodes $x_{n_{I}+1}$ and $x_{n_{II}+1}$ as goals for Player I and II,   respectively.

 \item Then, Theorem \ref{examples.main.thm.segment} yields the value of the game in all the nodes on the graph segment connecting  $x_{n_{I}+1}$ and $x_{n_{II}+1}$, in particular at $x_*$

 \item Once the value of the game at $x_*$ is known, Theorem \ref{examples.main.thm.segment} can be applied to all  remaining graph segments branching off $x_*$,   taking $x_*$ as a terminal node with terminal payoff $u(x_{*})$. As a result, the value of $u$ at all remaining nodes is obtained.

 \item Then, we check if the the function $u$ thus obtained satisfies the DPP at $x_{*}$. If so, it is the unique value of the game, and we have found a strategy that realizes it via Theorem \ref{examples.main.thm.segment}. If not, we go back to step (1), select a different pair of terminal nodes and iterate until we find the value of the game.
 
\end{enumerate}

\begin{figure}
\centering
\begin{tikzpicture}[scale=.8]
	\draw[semithick] (1.5,1) -- (0,0);
 \draw[fill] (1.5,1) circle [radius=3pt];
 \draw (1.5,1.75) node {};
 \draw[semithick] (1.5,1) -- (2,1);
 \draw[loosely dotted, semithick] (2.5,1) -- (4,1);
 \draw[semithick] (4.5,1) -- (5,1);
 \draw[fill] (5,1) circle [radius=3pt];
 \draw (5,1.75) node {};
 \draw[semithick] (5,1) -- (6.5,1);
 \draw[fill] (6.5,1) circle [radius=5pt];
 \draw (6.5,1.75) node {$x_{n_3+1}$}; 
 \draw[semithick] (1.5,-1) -- (0,0);
 \draw[fill] (1.5,-1) circle [radius=3pt];
 \draw (1.5,-0.25) node {};
 \draw[semithick] (1.5,-1) -- (2,-1);
 \draw[loosely dotted, semithick] (2.5,-1) -- (3,-1);
 \draw[semithick] (3.5,-1) -- (4,-1);
 \draw[fill] (4,-1) circle [radius=3pt];
 \draw (4,-0.25) node {};
 \draw[semithick] (4,-1) -- (5.5,-1);
 \draw[fill] (5.5,-1) circle [radius=5pt];
 \draw (5.5,-0.25) node {$x_{n_{2}+1}$}; 
 \draw[fill] (0,0) circle [radius=3pt];
 \draw (0,0.75) node {$x_*$};
 \draw[semithick] (0,0) -- (-1.5,0);
 \draw[fill] (-1.5,0) circle [radius=3pt];
 \draw (-1.5,0.75) node {};
 \draw[semithick] (-1.5,0) -- (-2,0);
 \draw[loosely dotted, semithick] (-2.5,0) -- (-3,0);
 \draw[semithick] (-3.5,0) -- (-4,0);
 \draw[fill] (-4,0) circle [radius=3pt];
 \draw (-4,0.75) node {};
 \draw[semithick] (-4,0) -- (-5.5,0);
 \draw[fill] (-5.5,0) circle [radius=5pt];
 \draw (-5.5,0.75) node {$x_{n_1+1}$};
\end{tikzpicture}
\caption{Game positions of the {\it Y}-game with multiple running nodes.}
\label{Ygame_MultN}
\end{figure}
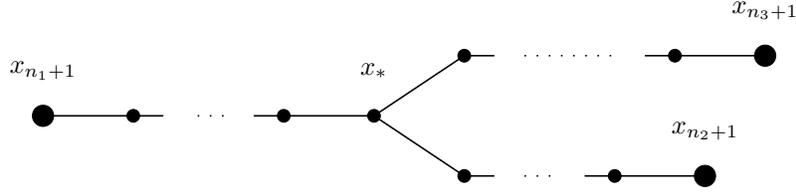


\section{The relation between the Totalitarian Tug-of-War and Jensen's extremal equations}\label{limit.eqs.jen}

In this section we are going to clarify the relation between the Totalitarian Tug-of-War and Jensen's extremal equations \eqref{eq.jensen.min.intro} and \eqref{eq.Jensen.max.intro}, see \cite{jensen1993}.

More precisely, in this section we are going to consider a Totalitarian Tug-of-War played in $\Omega\subset \mathbb{R}^n$, where the players can move the game token from $y\in\Omega$ to any position $x\in\overline{B}_\epsilon(y)$, with $\epsilon>0$ a parameter fixed at the beginning of the game.  In this case, the payment that Player I receives from
Player II when the latter is forced to choose the new game position is proportional to the 
step size $\epsilon$, with  proportionality constant $\lambda$.
Without loss of generality, we will consider $\lambda=1$ in the sequel.

As before, considering the two possible choices for Player I and then applying conditional probabilities for the coin toss we obtain the  Dynamic Programming Principle associated to this game, i.e.,
\begin{equation*}
u_\epsilon(x)=\max\left\{\inf_{y\in\overline{B}_\epsilon(x)\cap\overline{\Omega}}u_\epsilon(y)+\epsilon,\, \frac{1}{2}\left(\sup_{y\in\overline{B}_\epsilon(x)\cap\overline{\Omega}}u_\epsilon(y)+\inf_{y\in\overline{B}_\epsilon(x)\cap\overline{\Omega}}u_\epsilon(y)\right)\right\}
\end{equation*}
for all $x\in\Omega$. This can be equivalently written as
\begin{equation}\label{DiscreteDPP}
\begin{split}
\min\bigg\{u_\epsilon(x)-&\inf_{y\in\overline{B}_\epsilon(x)\cap\overline{\Omega}}u_\epsilon(y)-\epsilon,
\\
& u_\epsilon(x)-\frac{1}{2}\bigg(\sup_{y\in\overline{B}_\epsilon(x)\cap\overline{\Omega}}u_\epsilon(y)+\inf_{y\in\overline{B}_\epsilon(x)\cap\overline{\Omega}}u_\epsilon(y)\bigg)\bigg\}=0
\end{split}
\end{equation}
for all $x\in\Omega$. For simplicity, we  denote $\mathcal{G}[u_\epsilon](x)=0$ with
\begin{equation}\label{DiscreteDPP_operator}
\begin{split}
\mathcal{G}[u_\epsilon](x)=\min\bigg\{u_\epsilon(x)-&\inf_{y\in\overline{B}_\epsilon(x)\cap\overline{\Omega}}u_\epsilon(y)-\epsilon,
\\
& u_\epsilon(x)-\frac{1}{2}\bigg(\sup_{y\in\overline{B}_\epsilon(x)\cap\overline{\Omega}}u_\epsilon(y)+\inf_{y\in\overline{B}_\epsilon(x)\cap\overline{\Omega}}u_\epsilon(y)\bigg)\bigg\}.
\end{split}
\end{equation}

The Dirichlet problem that results from the combination of (\ref{DiscreteDPP}) and the terminal boundary payoff given by  a bounded function  $F:\partial\Omega\to\mathbb{R}$, corresponds~to
\begin{equation}\label{(30)[Rossi]}
\begin{cases}
\mathcal{G}[u_\epsilon](x)=0, & x\in\Omega;\\
u_\epsilon(x)=F(x), & x\in\partial\Omega.
\end{cases}
\end{equation}

Then we have the following result.

\begin{proposition} \label{prop.limit.eqn}
Let $u_\epsilon$ be the solution of  (\ref{(30)[Rossi]}). Assume that  there exists a function $u$ such that $u_\epsilon\to u$ uniformly in $\overline{\Omega}$ as $\epsilon\to0$. Then, $u$ is a viscosity solution to 
\begin{equation}\label{(33)[Rossi].limit.only}
\begin{cases}
\min\left\{|\nabla u|-1,-\Delta_\infty^N u\right\}=0, &  \textrm{in} \ \Omega;\\[0.25em]
u=F, & \textrm{on} \ \partial\Omega.
\end{cases}
\end{equation}
\end{proposition}

\begin{remark}
Considering a Totalitarian Tug-of-War which favors Player II instead of Player I, we can treat \eqref{eq.Jensen.max.intro} in a similar way, with equation \eqref{DiscreteDPP}  replaced by
\[
\begin{split}
\max\bigg\{
&
u_\epsilon(x)-\sup_{y\in\overline{B}_\epsilon(x)\cap\overline{\Omega}}u_\epsilon(y)+\lambda\,\epsilon ,\\
&\hspace{50pt}
u_\epsilon(x)-\frac{1}{2}\bigg(\sup_{y\in\overline{B}_\epsilon(x)\cap\overline{\Omega}}u_\epsilon(y)+\inf_{y\in\overline{B}_\epsilon(x)\cap\overline{\Omega}}u_\epsilon(y)\bigg)\bigg\}
=0.
\end{split}
\]
\end{remark}

\begin{remark}
We would like to   point out  an interesting connection with the numerical analysis of equations  \eqref{eq.jensen.min.intro} and \eqref{eq.Jensen.max.intro}.
More precisely, equations \eqref{eq.jensen.min.intro} and \eqref{eq.Jensen.max.intro} can be respectively approximated by the following schemes
\begin{equation}\label{disc.scheme.lower}
\begin{split}
\min\Bigg\{
\frac{1}{\epsilon}&\left(
u(x)-\inf_{y\in\overline{B}_\epsilon(x)\cap\overline{\Omega}}u(y)-\epsilon \lambda
\right),\\
&\hspace{50pt}\frac{1}{\epsilon^2}\left(2 u(x)-\sup_{y\in\overline{B}_\epsilon(x)\cap\overline{\Omega}}u(y)-\inf_{y\in\overline{B}_\epsilon(x)\cap\overline{\Omega}}u(y)
\right)\Bigg\}=0
\end{split}
\end{equation}
and
\begin{equation}\label{disc.scheme.upper}
\begin{split}
\max\Bigg\{
\frac{1}{\epsilon}&\left(
u(x)-\sup_{y\in\overline{B}_\epsilon(x)\cap\overline{\Omega}}u(y)+\epsilon \lambda
\right),\\
&\hspace{50pt}\frac{1}{\epsilon^2}\left(2 u(x)-\sup_{y\in\overline{B}_\epsilon(x)\cap\overline{\Omega}}u(y)-\inf_{y\in\overline{B}_\epsilon(x)\cap\overline{\Omega}}u(y)
\right)\Bigg\}=0,
\end{split}
\end{equation}
which are discrete elliptic in the sense of \cite{Oberman2006} (and, therefore, monotone in the sense of \cite{barles1991}) 
Furthermore, in a similar way to the Taylor expansion arguments in the proof of Proposition \ref{prop.limit.eqn}, one can  show that schemes \eqref{disc.scheme.lower} and \eqref{disc.scheme.upper} are   consistent (see \cite[Section 2]{barles1991} for the  definition). This means, roughly speaking, that the finite-difference operator converges in the viscosity sense towards the continuous operator of the PDE  as $\epsilon\to0$. 
Monotonicity and consistency, altogether with stability are important requirements for convergence, as established in the seminal paper \cite{barles1991}. Informally, the authors in \cite{barles1991} proved that any monotone, stable, and consistent scheme converges  provided that the limiting equation satisfies a type of comparison principle known as ``strong uniqueness property", which is usually difficult to prove. 
It seems an interesting question to tackle the convergence of schemes \eqref{disc.scheme.lower} and \eqref{disc.scheme.upper} and their numerical implementation but  we will not discuss it here.
\end{remark}

We include next for the reader's convenience the definition of viscosity solution for equation (\ref{(33)[Rossi].limit.only}).

\begin{definition}\label{viscosity_sltn}
A \emph{viscosity subsolution} of the equation \eqref{(30)[Rossi]} in $\Omega$ is an upper semicontinuous function $u:\Omega\to\mathbb{R}$ such that
\begin{equation}\label{(1.20)[Ju]}
\min\left\{|\nabla\varphi(\hat{x})|-1,-\Delta_\infty^N\varphi(\hat{x})\right\}\leq0,
\end{equation}
whenever $\hat{x}\in\Omega$ and $\varphi\in\mathcal{C}^2(\Omega)$ are such that $u(\hat{x})=\varphi(\hat{x})$ and $u(x)\leq\varphi(x)$, for all $x$ in a neighborhood of $\hat{x}$ (in other words, $\varphi$ touches $u$ at $\hat{x}$ from above in a neighborhood of $\hat{x}$, or equivalently, $u-\varphi$ has a local maximum at $\hat{x}$).

Similarly, a \emph{viscosity supersolution} of \eqref{(30)[Rossi]} in $\Omega$ is a lower semicontinuous function $u:\Omega\to\mathbb{R}$ such that
\begin{equation}\label{(1.21)[Ju]}
\min\left\{|\nabla\phi(\hat{x})|-1,-\Delta_\infty^N\phi(\hat{x})\right\}\geq0,
\end{equation}
whenever $\hat{x}\in\Omega$ and $\phi\in\mathcal{C}^2(\Omega)$ are such that $u(\hat{x})=\phi(\hat{x})$ and $u(x)\geq\phi(x)$, for all $x$ in a neighborhood of $\hat{x}$ (in other words, $\phi$ touches $u$ at $\hat{x}$ from below in a neighborhood of $\hat{x}$, or equivalently, $u-\phi$ has a local minimum at $\hat{x}$).

Finally, a function $u:\Omega\to\mathbb{R}$ is a \emph{viscosity solution} of \eqref{(30)[Rossi]} in $\Omega$ if it is both a viscosity subsolution and viscosity supersolution.
\end{definition}

The following lemma, which relies on uniform convergence, is needed in the proof of Proposition  \ref{prop.limit.eqn}. A proof for continuous functions
can be found 
 in \cite[Lemma 4.5]{Lindqvist2014}, but that version does not apply in our case because $u_\epsilon$ are not continuous in general.

\begin{lemma}\label{Lemma4.5[Lindqvist]}
Let $u=\lim_{\epsilon\to0}u_\epsilon$ uniformly in $\overline{\Omega}\subset\mathbb{R}^n$, $\hat{x}\in\Omega$ and $\varphi\in\mathcal{C}^2(\Omega)$ such that $u(\hat{x})=\varphi(\hat{x})$ and $u(x)<\varphi(x)$ for all $x$ in a neighborhood $\mathcal{U}$ of $\hat{x}$, when $x\neq\hat{x}$ (in other words, $\varphi$ touches $u$ at $\hat{x}$ strictly from above in $\mathcal{U}$, or equivalently, $u-\varphi$ has a strict maximum at $\hat{x}$ in $\mathcal{U}$). Then, for any given $\eta_\epsilon>0$ there exists a sequence of points $\{x_\epsilon\}_\epsilon\subset\mathcal{U}$ satisfying $\hat{x}=\lim_{\epsilon\to0}x_\epsilon$ such that
\begin{equation}\label{(14)[Rossi]_subsltn}
u_\epsilon(x)-\varphi(x)\leq u_\epsilon(x_\epsilon)-\varphi(x_\epsilon)+\eta_\epsilon\quad\textnormal{for all}\ \ x\in\mathcal{U}.
\end{equation}

Moreover, whenever $\phi\in\mathcal{C}^2(\Omega)$ touches $u$ at $\hat{x}$ strictly from below in $\mathcal{U}$, for any given $\eta_\epsilon>0$ there is a sequence of points $\{x_\epsilon\}_\epsilon\subset\mathcal{U}$ satisfying $\hat{x}=\lim_{\epsilon\to0}x_\epsilon$, such that
\begin{equation}\label{(14)[Rossi]_supersltn}
u_\epsilon(x)-\phi(x)\geq u_\epsilon(x_\epsilon)-\phi(x_\epsilon)-\eta_\epsilon\quad\textnormal{for all}\ \ x\in\mathcal{U}.
\end{equation}
\end{lemma}

\begin{proof}
Let  $B_r(\hat{x})$ be a fixed, small ball and write
 $u_\epsilon-\varphi=(u_\epsilon-u)+(u-\varphi)$. Since $u_\epsilon\to u$,
we have that
\[
\begin{split}
\sup_{\mathcal{U}\backslash B_r(\hat{x})}(u_\epsilon-\varphi) &
\leq \sup_{\mathcal{U}\backslash B_r(\hat{x})}(u_\epsilon-u)+\sup_{\mathcal{U}\backslash B_r(\hat{x})}(u-\varphi)\\
 &\leq \frac12\sup_{\mathcal{U}\backslash B_r(\hat{x})}(u-\varphi)<0,
 \end{split}
\]
 for $\epsilon$ small enough.  Moreover, for  $\epsilon<\epsilon_r$ small enough
\begin{equation*}
\sup_{\mathcal{U}\backslash B_r(\hat{x})}(u_\epsilon-\varphi)<u_\epsilon(\hat{x})-\varphi(\hat{x})
\end{equation*}
 because the right-hand side approaches zero as $\epsilon\to0$ 
while the left-hand side has a strictly negative limit. 
Notice that this means the value of $u_\epsilon-\varphi$ at the center is larger than the supremum over $\mathcal{U}\backslash B_r(\hat{x})$.
Therefore, there is a point $x_\epsilon\in B_r(\hat{x})$ such that
\begin{equation*}
\sup_{\mathcal{U}}(u_\epsilon-\varphi)\leq u_\epsilon(x_\epsilon)-\varphi(x_\epsilon)+\eta_\epsilon
\end{equation*}
when $\epsilon<\epsilon_r$.
The proof finishes by letting $r\to0$ via a sequence, say $r=1,\frac{1}{2},\frac{1}{3},\dots$ Hence, there is a sequence of points $x_\epsilon\to\hat{x}$ which satisfy (\ref{(14)[Rossi]_subsltn}), as desired. The proof of the supersolution case follows similarly.
\end{proof}

We are now in position to prove Proposition  \ref{prop.limit.eqn}.

\begin{proof}[Proof of Proposition  \ref{prop.limit.eqn}] 
\noindent1.\quad{}We will show first that $u$ is a viscosity supersolution  of \eqref{(30)[Rossi]}. Let $\hat{x}\in\Omega$ and  $\phi\in\mathcal{C}^2(\Omega)$ such that $\phi$ touches $u$ at $\hat{x}$ strictly from below in a neighborhood of $\hat{x}$. Our goal is to prove that
\begin{equation}\label{goal.limitPDE.supersol} 
\min\left\{|\nabla \phi(\hat{x})|-1,-\Delta_\infty^N \phi(\hat{x})\right\}\geq0. 
\end{equation}

Lemma \ref{Lemma4.5[Lindqvist]} holds with $\eta_\epsilon=\epsilon^3$ and we know there exists a sequence of points $x_\epsilon\to \hat{x}$ such that \eqref{(14)[Rossi]_supersltn} holds. Rearranging terms we  deduce 
\begin{equation}\label{eqn.with.sup1}
\phi(x_\epsilon)-\max_{\overline{B}_\epsilon(x_\epsilon)}\phi+\epsilon^3\geq u_\epsilon(x_\epsilon)-\sup_{\overline{B}_\epsilon(x_\epsilon)}u_\epsilon
\end{equation}
and
\begin{equation}\label{eqn.with.inf1}
\phi(x_\epsilon)
-\min_{\overline{B}_\epsilon(x_\epsilon)}\phi+\epsilon^3\geq u_\epsilon(x_\epsilon)-\inf_{\overline{B}_\epsilon(x_\epsilon)}u_\epsilon.
\end{equation}

\medskip

\noindent2.\quad{} 
We deduce from (\ref{DiscreteDPP}), that
\[
u_\epsilon(x)-\inf_{\overline{B}_\epsilon(x)}u_\epsilon-\epsilon\geq0\qquad\textnormal{and}\qquad u_\epsilon(x)-\frac{1}{2}\left(\sup_{\overline{B}_\epsilon(x)}u_\epsilon+\inf_{\overline{B}_\epsilon(x)}u_\epsilon\right)\geq0
\]
with equality in at least one of the two equations. From the first one and \eqref{eqn.with.inf1}, we get
\begin{equation}\label{eqn.first.contradiciton.supersol}
\phi(x_\epsilon)-\min_{\overline{B}_\epsilon(x_\epsilon)}\phi
+
\epsilon^3\geq u_\epsilon(x_\epsilon)-\inf_{\overline{B}_\epsilon(x_\epsilon)}u_\epsilon
\geq
\epsilon.
\end{equation}

Let us show that $|\nabla\phi(\hat x)|\geq1$. Assume to the contrary that $|\nabla\phi(\hat x)|<1$ and  write $\min_{\overline{B}_\epsilon(x_\epsilon)}\phi=\phi(x_\epsilon-\epsilon v_\epsilon)$ for some $v_\epsilon\in \overline{B}_1(0)$. Then, a first-order Taylor expansion yields
\begin{equation}\label{eqn.first.contradiciton.supersol.2}
\frac{1}{\epsilon}\bigg(
\phi(x_\epsilon)-\min_{\overline{B}_\epsilon(x_\epsilon)}\phi
\bigg)
=
 \nabla \phi(x_\epsilon)\cdot v_\epsilon+o(1)\leq|\nabla \phi(x_\epsilon)| +o(1) <1
\end{equation}
for $\epsilon$ small enough,
a contradiction with \eqref{eqn.first.contradiciton.supersol}.

\medskip

\noindent3.\quad{}Let  $x_\epsilon^{\min},x_\epsilon^{\max}\in \overline{B}_\epsilon(x_\epsilon)$ be  such that
 $\phi(x_\epsilon^{\min})=\min_{\overline{B}_\epsilon(x_\epsilon)}\phi$ 
and 
$\phi(x_\epsilon^{\max})=\max_{\overline{B}_\epsilon(x_\epsilon)}\phi$, respectively. We can show that
\begin{equation}\label{minimum_point}
x_\epsilon^{\min}=x_\epsilon-\epsilon\left[\frac{\nabla\phi(x_\epsilon)}{|\nabla\phi(x_\epsilon)|}+o(1)\right]\quad\textnormal{as}\ \ \epsilon\to0,
\end{equation}
and
\begin{equation}\label{maximum_point}
x_\epsilon^{\max}=x_\epsilon+\epsilon\left[\frac{\nabla\phi(x_\epsilon)}{|\nabla\phi(x_\epsilon)|}+o(1)\right]\quad\textnormal{as}\ \ \epsilon\to0.
\end{equation}
We will provide the details of the proof of \eqref{minimum_point} because \eqref{maximum_point} follows similarly.

First, notice that  $x_\epsilon^{\min}\in\partial B_\epsilon(x_\epsilon)$. To see this, assume to the contrary that there exists a subsequence $x_{\epsilon_k}^{\min}\in B_{\epsilon_k}(x_{\epsilon_k})$ of minimum points of $\phi$ in $\Omega$. Then $\nabla\phi(x_{\epsilon_k}^{\min})=0$ and since $\lim_{\epsilon_k\to0}x_{\epsilon_k}^{\min}=\hat{x}$, the continuity of $\phi$ implies that $\nabla\phi(\hat{x})=0$, a contradiction.

Then, we can write $x_\epsilon^{\min}=x_\epsilon-\epsilon\, v$ for $|v|=1$, and let $\omega$ be any fixed direction with $|\omega|=1$. Since $\phi(x_\epsilon^{\min})\leq\phi(x_\epsilon-\epsilon\, \omega)$, a Taylor expansion of $\phi$ around $x_\epsilon$ gives 
\begin{equation*}
\phi(x_\epsilon)-\epsilon\left<\nabla\phi(x_\epsilon), v\right>+o(\epsilon)=\phi(x_\epsilon^{\min})\leq\phi(x_\epsilon-\epsilon\, \omega)\quad\textnormal{as}\ \ \epsilon\to0,
\end{equation*}
or equivalently,
\begin{equation*}
\left<\nabla\phi(x_\epsilon),\, v\right>+o(1)\geq\frac{-\phi(x_\epsilon-\epsilon\, \omega)+\phi(x_\epsilon)}{\epsilon}=\left<\nabla\phi(x_\epsilon),\, \omega\right>+o(1)\quad\textnormal{as}\ \ \epsilon\to0.
\end{equation*}
Since the previous argument holds for any direction $\omega$, we can conclude that
\begin{equation*}
v=\frac{\nabla\phi(x_\epsilon)}{|\nabla\phi(x_\epsilon)|}+o(1)\quad\textnormal{as}\ \ \epsilon\to0.
\end{equation*}
and the proof of   \eqref{minimum_point} is complete.

%

\medskip

\noindent4.\quad{}Now, assume 
\[
u_\epsilon(x)-\frac{1}{2}\left(\sup_{\overline{B}_\epsilon(x)}u_\epsilon+\inf_{\overline{B}_\epsilon(x)}u_\epsilon\right)\geq0. 
\]
Adding \eqref{eqn.with.sup1} and \eqref{eqn.with.inf1}, and using this last inequality yields
\begin{equation}\label{second.part.holds.supersol}
\begin{split}
\phi(x_\epsilon)&-\frac{1}{2}\left(\max_{\overline{B}_\epsilon(x_\epsilon)}\phi+\min_{\overline{B}_\epsilon(x_\epsilon)}\phi\right)+\epsilon^3
\\
&\geq u_\epsilon(x_\epsilon)-\frac{1}{2}\left(\sup_{\overline{B}_\epsilon(x_\epsilon)}u_\epsilon+\inf_{\overline{B}_\epsilon(x_\epsilon)}u_\epsilon\right)\geq0.
\end{split}
\end{equation}

On the other hand, let $\tilde{x}_\epsilon^{\min}=2x_\epsilon-x_\epsilon^{\min}$ be the symmetric point of $x_\epsilon^{\min}$ with respect to $x_\epsilon$. 
Then, 
\[
\phi(\tilde{x}_\epsilon^{\min})+\phi(x_\epsilon^{\min})-2\phi(x_\epsilon)\leq\max_{\overline{B}\epsilon(x_\epsilon)}\phi+\min_{\overline{B}\epsilon(x_\epsilon)}\phi-2\phi(x_\epsilon),
\]
and a second-order Taylor expansion of $\phi$ around $x_\epsilon$, gives
\begin{equation}\label{(19)[Rossi]_minPoint}
\begin{split}
\max_{\overline{B}\epsilon(x_\epsilon)}\phi&+\min_{\overline{B}\epsilon(x_\epsilon)}\phi-2\phi(x_\epsilon)
\\
&\geq
\left<D^2\phi(x_\epsilon)\, (x_\epsilon^{\min}-x_\epsilon),\, \left(x_\epsilon^{\min}-x_\epsilon\right)\right>+o(\epsilon^2)\quad\textnormal{as}\ \ \epsilon\to0.
\end{split}
\end{equation}
Then, the combination of \eqref{second.part.holds.supersol}, (\ref{(19)[Rossi]_minPoint}), and (\ref{minimum_point})  gives us
\begin{equation*}
-\epsilon^2\left<D^2\phi(x_\epsilon)\, \frac{\nabla\phi(x_\epsilon)}{|\nabla\phi(x_\epsilon)|},\frac{\nabla\phi(x_\epsilon)}{|\nabla\phi(x_\epsilon)|}\right>\geq o(\epsilon^2)\quad\textnormal{as}\ \ \epsilon\to0.
\end{equation*}
Dividing by $\epsilon^2$ and taking $\epsilon\to0$ we get  $-\Delta_\infty^N\phi(\hat{x})\geq0$, and the proof of \eqref{goal.limitPDE.supersol} is complete.

\medskip

\noindent5.\quad{}The proof that $u$ is a viscosity subsolution to \eqref{(30)[Rossi]}, is similar to the supersolution case. Let $\hat{x}\in\Omega$ and  $\varphi\in\mathcal{C}^2(\Omega)$ such that $\varphi$ touches $u$ at $\hat{x}$ strictly from above in a neighborhood of $\hat{x}$. Our goal is to prove
\begin{equation}\label{goal.limitPDE.subsol} 
\min\left\{|\nabla \varphi(\hat{x})|-1,-\Delta_\infty^N \varphi(\hat{x})\right\}\leq0. 
\end{equation}

Just as in the subsolution case, 
Lemma \ref{Lemma4.5[Lindqvist]} holds with $\eta_\epsilon=\epsilon^3$ 
and
we have that  there exists a sequence of points $x_\epsilon\to \hat{x}$ such that (\ref{(14)[Rossi]_subsltn}) holds. Rearranging terms we  deduce 
\begin{equation}\label{eqn.with.sup}
\varphi(x_\epsilon)-\max_{\overline{B}_\epsilon(x_\epsilon)}\varphi-\epsilon^3\leq u_\epsilon(x_\epsilon)-\sup_{\overline{B}_\epsilon(x_\epsilon)}u_\epsilon
\end{equation}
and
\begin{equation}\label{eqn.with.inf}
\varphi(x_\epsilon)
-\min_{\overline{B}_\epsilon(x_\epsilon)}\varphi-\epsilon^3\leq u_\epsilon(x_\epsilon)-\inf_{\overline{B}_\epsilon(x_\epsilon)}u_\epsilon.
\end{equation}

\medskip

\noindent6.\quad{}
Notice  that we can assume $\left|\nabla\varphi(\hat{x})\right|>1$ since otherwise \eqref{goal.limitPDE.subsol} holds
and the proof is complete.
In particular, $\nabla\varphi(x_\epsilon)\neq0$ for $\epsilon$ small enough and
we can prove \eqref{minimum_point} and  \eqref{maximum_point} just as before.

Moreover, 
\eqref{eqn.with.inf}, \eqref{minimum_point} and a Taylor expansion give 
\[
 \frac{1}{\epsilon}\bigg(
 u_\epsilon(x_\epsilon)-\inf_{\overline{B}_\epsilon(x_\epsilon)}u_\epsilon
\bigg)
\geq
 \frac{1}{\epsilon}\bigg(
\varphi(x_\epsilon)-\min_{\overline{B}_\epsilon(x_\epsilon)}\varphi
-
\epsilon^3
\bigg)
=
 |\nabla \varphi(x_\epsilon)|+o(1)
>
1
\]
for $\epsilon$ small enough.
Therefore, taking into account \eqref{DiscreteDPP}, we deduce that
\[
u_\epsilon(x_\epsilon)-\frac{1}{2}\left(\sup_{\overline{B}_\epsilon(x_\epsilon)}u_\epsilon+\inf_{\overline{B}_\epsilon(x_\epsilon)}u_\epsilon\right)=0
\]
and, adding  \eqref{eqn.with.sup} and \eqref{eqn.with.inf},  we obtain
\begin{equation}\label{eqn.this.last.ineualiyt}
\begin{split}
\varphi(x_\epsilon)&-\frac{1}{2}\left(\max_{\overline{B}_\epsilon(x_\epsilon)}\varphi+\min_{\overline{B}_\epsilon(x_\epsilon)}\varphi\right)-\epsilon^3
\\
&\leq u_\epsilon(x_\epsilon)-\frac{1}{2}\left(\sup_{\overline{B}_\epsilon(x_\epsilon)}u_\epsilon+\inf_{\overline{B}_\epsilon(x_\epsilon)}u_\epsilon\right)=0.
\end{split}
\end{equation}

Now, consider the point $\tilde{x}_\epsilon^{\max}=2x_\epsilon-x_\epsilon^{\max}$, the symmetric point of $x_\epsilon^{\max}$ with respect to $x_\epsilon$.
The second-order Taylor expansion of $\varphi$ around $x_\epsilon$ evaluated at $x_\epsilon^{\max}$ and $\tilde{x}_\epsilon^{\max}$ yields,
\begin{equation}\label{(19)[Rossi]_maxPoint}
\begin{split}
\max_{\overline{B}\epsilon(x_\epsilon)}\varphi&+\min_{\overline{B}\epsilon(x_\epsilon)}\varphi-2\varphi(x_\epsilon)
\leq
\varphi\left(x_\epsilon^{\max}\right)+\varphi\left(\tilde{x}_\epsilon^{\max}\right)-2\varphi(x_\epsilon)
\\
&=
\left<D^2\varphi(x_\epsilon)\, (x_\epsilon^{\max}-x_\epsilon),\, \left(x_\epsilon^{\max}-x_\epsilon\right)\right>+o\left(\epsilon^2\right)\quad\textnormal{as}\ \ \epsilon\to0.
\end{split}
\end{equation}

The combination of \eqref{maximum_point}, \eqref{eqn.this.last.ineualiyt}, and (\ref{(19)[Rossi]_maxPoint})  gives us 
\begin{equation*}
-\epsilon^2\left<D^2\varphi(x_\epsilon)\, \frac{\nabla\varphi(x_\epsilon)}{|\nabla\varphi(x_\epsilon)|},\frac{\nabla\varphi(x_\epsilon)}{|\nabla\varphi(x_\epsilon)|}\right>\leq o\left(\epsilon^2\right)\quad\textnormal{as}\ \ \epsilon\to0.
\end{equation*}
After dividing by $\epsilon^2$ and taking $\epsilon\to0$ we get $-\Delta_\infty^N\varphi(\hat{x})\leq0$ and \eqref{goal.limitPDE.subsol} holds as desired.
\end{proof}

\bibliographystyle{amsplain}
\bibliography{references}

\end{document}